\documentclass[12pt,reqno,a4paper]{amsart}

\usepackage{fullpage}

\usepackage[foot]{amsaddr}
\usepackage{graphicx,enumerate,nicefrac,xcolor,bm,cite}
\usepackage{amsmath, amssymb, amstext, amsfonts}
\usepackage{stmaryrd}
\usepackage{comment}
\usepackage{mathtools}
\usepackage{subfig}
\usepackage[utf8x]{inputenc}
\usepackage{algorithm,algpseudocode,tabularx}

\makeatletter
\newcommand{\multiline}[1]{%
  \begin{tabularx}{\dimexpr\linewidth-\ALG@thistlm}[t]{@{}X@{}}
    #1
  \end{tabularx}
}
\makeatother

\newcommand{\norm}[1]{\|#1\|}

\newcommand{\operator}[1]{\mathsf{#1}}
\newcommand{\A}{\operator{A}}

\newcommand{\F}{\operator{F}}   
\newcommand{\E}{\operator{E}}

\renewcommand{\H}{\operator{E}}

\renewcommand{\d}{\operator{d}}

\newcommand{\dt}{\,\d t}
\newcommand{\dx}{\,\d\bm x}
\newcommand{\jmp}[1]{\left\llbracket#1\right\rrbracket}

\newcommand{\dprod}[1]{\left<#1\right>}

\newcommand{\un}[1]{{u}_N^{#1}}

\newcommand{\dissol}[1]{u_{#1}^\star}

\newcommand{\twon}[2]{\norm{#1}_{L^2(#2)}}

\newcommand{\refine}{\mathtt{refine}}
\newcommand{\nN}{\n(N)}
\newcommand{\n}{\mathfrak{n}}
\newcommand{\uN}{\mathfrak{N}}

\newcommand{\abd}{\beta}
\newcommand{\aco}{\alpha}

\newcommand{\Flc}{L_\operator{F}}

\newcommand{\dpa}{\delta}
\newcommand{\Fsm}{\nu}

\newcommand{\cstb}{C_{\mathrm{stb}}}

\newcommand{\crel}{C_{\mathrm{rel}}}

\newcommand{\qred}{q_{\mathrm{red}}}
\newcommand{\cmark}{C_{\mathrm{mark}}}

\newcommand{\copt}{c_{\mathrm{opt}}}
\newcommand{\Copt}{C_{\mathrm{opt}}}

\newcommand{\qlin}{q_{\mathrm{lin}}}

\newcommand{\clin}{C_{\mathrm{lin}}}
\newcommand{\cson}{C_{\mathrm{ref}}}
\newcommand{\cmesh}{C_{\mathrm{mesh}}}
\newcommand{\Tref}{\mathcal{T}^{\rm ref}}
\newcommand{\Trefo}{\mathcal{T}'}
\newcommand{\nl}{\mathfrak{A}}
\DeclareMathOperator{\Div}{div}

\usepackage{fancyhdr}
\advance\footskip0.4cm
\advance\textheight-0.4cm
\calclayout
\newcommand*\patchAmsMathEnvironmentForLineno[1]{%
  \expandafter\let\csname old#1\expandafter\endcsname\csname #1\endcsname
  \expandafter\let\csname oldend#1\expandafter\endcsname\csname end#1\endcsname
  \renewenvironment{#1}%
     {\linenomath\csname old#1\endcsname}%
     {\csname oldend#1\endcsname\endlinenomath}}%
\newcommand*\patchBothAmsMathEnvironmentsForLineno[1]{%
  \patchAmsMathEnvironmentForLineno{#1}%
  \patchAmsMathEnvironmentForLineno{#1*}}%
\AtBeginDocument{%
\patchBothAmsMathEnvironmentsForLineno{equation}%
\patchBothAmsMathEnvironmentsForLineno{align}%
\patchBothAmsMathEnvironmentsForLineno{flalign}%
\patchBothAmsMathEnvironmentsForLineno{alignat}%
\patchBothAmsMathEnvironmentsForLineno{gather}%
\patchBothAmsMathEnvironmentsForLineno{multline}%
}
\usepackage[mathlines]{lineno}

\makeatletter
\def\@seccntformat#1{\hspace*{4mm}%
  \protect\textup{\protect\@secnumfont
    \ifnum\pdfstrcmp{subsection}{#1}=0 \bfseries\fi
    \csname the#1\endcsname
    \protect\@secnumpunct
  }%
}
\makeatother

\newtheorem{theorem}{Theorem}[section]

\newtheorem{proposition}[theorem]{Proposition}

\theoremstyle{definition}

\newtheorem{remark}[theorem]{Remark}

\title[Adaptive iterative linearized FEM]{Energy contraction and optimal convergence of adaptive iterative linearized finite element methods}

\author[P.~Heid]{Pascal Heid$^1$}
\email{pascal.heid@maths.ox.ac.uk}

\author[D.~Praetorius]{Dirk Praetorius$^2$}
\email{dirk.praetorius@asc.tuwien.ac.at}

\author[T.~P.~Wihler]{Thomas P.~Wihler$^3$}
\email{wihler@math.unibe.ch}

\address[$^1$]{Mathematical Institute, University of Oxford, Woodstock Road, Oxford OX2 6GG, UK}

\address[$^2$]{TU Wien, Institute of Analysis and Scientific Computing, Wiedner Hauptstr.~8--10/E101/4, 1040 Wien, Austria}

\address[$^3$]{Mathematics Institute, University of Bern, Sidlerstr.~5, CH-3012 Bern, Switzerland}

\thanks{The authors acknowledge the financial support of the Swiss National Science Foundation (SNF), Grant No. 200021\underline{\space}182524, and of the Austrian Science Fund (FWF) Grant No.~SFB F65 and P33216.}



\begin{document}

\begin{abstract}
We revisit a unified methodology for the iterative solution of nonlinear equations in Hilbert spaces. Our key observation is that the general approach from \cite{HeidWihler:19v2,HeidWihler2:19v1} satisfies an energy contraction property in the context of (abstract) strongly monotone problems. This property, in turn, is the crucial ingredient in the recent convergence analysis in~\cite{GHPS:2020}. In particular, we deduce that adaptive iterative linearized finite element methods (AILFEMs) lead to full linear convergence with optimal algebraic rates with respect to the degrees of freedom as well as the total computational time.    
\end{abstract}

\maketitle

\def\TT{\mathcal{T}}
\def\exact{\star}
\def\dual#1#2{\langle#1\,,\,#2\rangle}
\def\XX{\mathcal{X}}
\def\N{\mathbb{N}}
\section{Introduction}

\noindent
This work deals with the effective numerical solution of quasi-linear boundary value problems of the type
\begin{align}\label{eq:strongform}
 \begin{split}
  - \Div \nl(\nabla u^\star) &= g \quad \text{in } \Omega,\\
 u^\star &= 0 \quad \text{ on } \partial\Omega,
 \end{split}
\end{align}
where $\Omega \subset \mathbb{R}^d$, $d\ge 2$, is a bounded Lipschitz domain with polytopic boundary. More precisely, given $g \in L^2(\Omega)$ and a strongly monotone and Lipschitz continuous nonlinearity $\nl \colon \mathbb{R}^d \to \mathbb{R}^d$, we aim to analyze an adaptive loop 
\begin{align}\label{eq:semr}
 \boxed{\texttt{~solve~\&~estimate~}}
 \longrightarrow
 \boxed{\texttt{~mark~}}
 \longrightarrow
 \boxed{\texttt{~refine~}}
\end{align}
for the numerical approximation, at optimal cost, of the weak solution $u^\star \in H^1_0(\Omega)$ of~\eqref{eq:strongform}, with $H^1_0(\Omega)$ being the standard Sobolev space of $H^1$-functions on~$\Omega$ with zero trace along~$\partial\Omega$.

Based on a triangulation $\TT_N$ of $\Omega$, which will be automatically refined by the adaptive loop~\eqref{eq:semr} in order to resolve the possible singularities of $u^\exact$, and a corresponding $H^1$-conforming FEM space $X_N$, the discrete variational formulation reads: 
Find $u_N^\star \in X_N$ such that
\begin{align}\label{eq:discrete}
 (\nl(\nabla u_N^\exact),\nabla v_N)_{L^2(\Omega)} = (g,v_N)_{L^2(\Omega)}
 \qquad \forall v_N \in X_N,
\end{align}
where $(\cdot,\cdot)_{L^2(\Omega)}$ denotes the $L^2(\Omega)$-inner product. We emphasize that the discrete formulation~\eqref{eq:discrete} is nonlinear, so that $u_N^\exact$ cannot be computed exactly in general. For this reason we linearize~\eqref{eq:discrete} iteratively, thereby  providing approximations $u_N^n \approx u_N^\exact$, for $n\ge 0$, which are obtained by solving certain underlying discrete linear systems.

Overall, the adaptive loop~\eqref{eq:semr} monitors and terminates the iterative linearization of the discrete nonlinear problem and steers the local mesh-refinement, hence giving rise to an \emph{adaptive iterative linearized finite element method} (AILFEM). In each step, the algorithm solves only one discrete linear system; for the ease of presentation, we will assume that this ``linear solve'' is performed exactly. 

Optimal convergence of adaptive algorithms for second-order elliptic problems is well-understood nowadays; see, e.g.,~\cite{doerfler1996,mns2000,bdd2004,stevenson2007,ckns2008,ffp2014}) for linear problems,~\cite{veeser2002,dk2008,bdk2012,MR4091593} for the $p$-Laplacian, \cite{cc2008,cc2009,cc2013} for convex minimization problems, \cite{GarauMorinZuppa:2012,ffp2014,GantnerHaberlPraetoriusStiftner:17,GHPS:2020} for the present setting, and~\cite{CarstensenFeischlPagePraetorius:14} for a general framework. Some works also account for the approximate computation of the discrete solutions by iterative (and hence inexact) solvers; see, e.g.,~\cite{bms2010,agl2013,GHPS:2020} for linear problems and~\cite{gmz2011,GantnerHaberlPraetoriusStiftner:17,HeidWihler:19v2,HeidWihler2:19v1,GHPS:2020} for nonlinear model problems. Moreover, already the seminal work~\cite{stevenson2007} has addressed the optimal computational cost of adaptive FEM for the Poisson model problem under realistic assumptions on a non-specified inexact solver and a similar result is given in~\cite{cc2012} for an adaptive Laplace eigenvalue solver.
Finally, there are various papers on \textsl{a~posteriori} error estimation which also include the iterative and inexact solution for nonlinear problems; see, e.g.,~\cite{eev2011,ev2013,aw2015,CongreveWihler:17,HeidWihler:19v2} and the references therein. 

While the AILFEM approach can be employed for more general nonlinear problems than in the present setting (see, e.g.,~\cite{ev2013} for empirical results with optimal convergence), we note that the thorough mathematical understanding is widely open. In the context of the current article, we point to the recent work~\cite{GantnerHaberlPraetoriusStiftner:17}, where 
AILFEM for the solution of~\eqref{eq:strongform} with strongly monotone nonlinearity has been analyzed in the specific context of the Zarantonello linearization approach originally proposed in~\cite{CongreveWihler:17}. Moreover, the results of~\cite{GantnerHaberlPraetoriusStiftner:17} have been extended to contractive iterative solvers in~\cite{GHPS:2020}. In particular, it has been shown 
that AILFEM based on the Zarantonello linearization approach leads to an optimal convergence rate with respect to the number of degrees of freedom as well as with respect to the overall computational cost.

The purpose of the present work is to show that the optimality results from~\cite{GHPS:2020} 
carry over to AILFEM beyond the Zarantonello linearization approach and also include, in particular, the adaptive linearization by the Ka\v{c}anov iteration or the damped Newton method. To this end, the key property is an energy contraction, which constitutes the crucial ingredient in the analysis of~\cite{GHPS:2020} and thus allows to combine and extend the recent developments from~\cite{GantnerHaberlPraetoriusStiftner:17,GHPS:2020} and~\cite{HeidWihler:19v2,HeidWihler2:19v1}.

\textbf{Outline.}
%
This work is organized as follows: In Section~\ref{section:ilg}, the quasi-linear boundary value problem~\eqref{eq:strongform} is cast into an abstract Hilbert space setting, and the iterative linearized Galerkin method from~\cite{HeidWihler:19v2,HeidWihler2:19v1} is applied. We notice from~\cite{HeidWihler:19v2} that this framework covers the Zarantonello and Ka\v{c}anov iteration as well as the damped Newton method. The new core argument of our paper is Theorem~\ref{prop:energyreduction}, which provides an energy contraction property that is crucial for the application of the results from~\cite{GHPS:2020}. In the spirit of~\cite{CarstensenFeischlPagePraetorius:14}, Section~\ref{section:ailfem} formulates abstract assumptions on the FEM discretization and on the \textsl{a~posteriori} error estimators, and then states the AILFEM loop (Algorithm~\ref{alg:praetal}). Section~\ref{sc:convergence} recalls the main results from~\cite{GHPS:2020}, which now allow to generalize the analysis of AILFEM beyond the Zarantonello linearization approach. Moreover, we discuss in detail how these results extend those of~\cite{GHPS:2020,HeidWihler2:19v1}. Numerical experiments in Section~\ref{sec:examples} underpin the theoretical findings. Some brief conclusion is drawn in Section~\ref{section:conclusion}.

\section{Iterative linearized Galerkin approach}\label{section:ilg}

\subsection{Abstract model problem}

On a real Hilbert space $X$ with inner product~$(\cdot,\cdot)_X$ and induced norm~$\|\cdot\|_X$, we consider a nonlinear operator~$\F:\,X\to X^\star$, where $X^\star$ denotes the dual space of $X$. 
We consider the following nonlinear operator equation:
\begin{equation}\label{eq:F=0}
\F(u)=0 \quad\text{in }X^\star.
\end{equation}
With $\dprod{\cdot,\cdot}$ signifying the duality pairing on~$X^\star\times X$, the weak form of this problem reads:
\begin{equation}\label{eq:F=0weak}
\text{Find } u\in X \text{ such that} \quad  \dprod{\F(u),v}=0\qquad \forall v\in X.
\end{equation}
For the purpose of this work, we suppose that $\F$ satisfies the following conditions: 
\begin{enumerate}[(F1)]
\item \emph{Lipschitz continuity:} There exists a constant $\Flc>0$ such that 
\begin{align*}
\left|\dprod{\F(u)-\F(v),w}\right| \leq \Flc \norm{u-v}_X \norm{w}_X \qquad \forall u,v,w \in X.
\end{align*}
\item \emph{Strong monotonicity:} There exists a constant $\Fsm>0$ such that 
\begin{align*}
 \Fsm \norm{u-v}_X^2 \leq \dprod{\F(u)-\F(v),u-v} \qquad \forall u,v  \in X.
\end{align*}
\end{enumerate}
Assuming (F1) and (F2), the main theorem of strongly monotone operators guarantees that~\eqref{eq:F=0} (or equivalently~\eqref{eq:F=0weak}) has a unique solution $u^\star \in X$; see, e.g., \cite[\S3.3]{Necas:86} or \cite[Thm.~25.B]{Zeidler:90}.

\subsection{Iterative linearization}

Following the recent approach~\cite{HeidWihler:19v2}, we apply a general fixed point iteration scheme for the solution of~\eqref{eq:F=0}. For given~$v\in X$, we consider a linear and invertible \emph{preconditioning operator} $\A[v] \in \mathcal{L}(X, X^\star)$. Then, the nonlinear problem~\eqref{eq:F=0} is equivalent to $\A[u]u = f(u)$, where $f(u):= \A[u]u-\F(u)$.
For any suitable initial guess~$u^0\in X$, this leads to the following iterative scheme:
\begin{equation}\label{eq:fp0}
\text{Find } u^{n+1}\in X \text{ such that} \quad \A[u^n]u^{n+1}=f(u^n) \quad \text{in } X^\star \qquad \forall n\ge 0.
\end{equation}
Note that the above iteration is a \emph{linear equation} for~$u^{n+1}$, thereby rendering~\eqref{eq:fp0} an \emph{iterative linearization scheme} for~\eqref{eq:F=0}. Given $u^n\in X$, the weak form of~\eqref{eq:fp0} is based on the bilinear form $a(u^n;v,w):=\dprod{\A[u^n]v,w}$, for $v,w\in X$, and the solution~$u^{n+1}\in X$ of~\eqref{eq:fp0} can be obtained from 
\begin{equation}\label{eq:itweak}
a(u^n;u^{n+1},w)=\dprod{f(u^n),w}\qquad \forall w\in X.
\end{equation}
Throughout, for any~$u\in X$, we suppose that the bilinear form~$a(u;\cdot,\cdot)$ is uniformly coercive and bounded, i.e., there are two constants~$\alpha,\beta>0$ independent of $u \in X$ such that
\begin{equation}\label{eq:coercive}
a(u;v,v) \geq \aco \|v\|_X^2 \qquad\forall v \in X,
\end{equation}
and
\begin{equation}\label{eq:continuity}
a(u;v,w) \leq \abd \norm{v}_X \norm{w}_X \qquad\forall  v,w \in X.
\end{equation}
For any given~$u^n\in X$, owing to the Lax--Milgram theorem, these two properties imply that the linear equation~\eqref{eq:itweak} admits a unique solution~$u^{n+1}\in X$.

\subsection{Iterative linearized Galerkin approach (ILG) and energy contraction}\label{sc:Y}

In order to cast~\eqref{eq:itweak} into a computational framework, we consider 
closed (and mainly finite dimensional)
subspaces $Y\subseteq X$, endowed with the inner product and norm on $X$. 
Denote by $u_Y^\star \in Y$ the unique solution of the equation
\begin{align} \label{eq:F=0Y}
\dprod{\F(u^\star_Y),v}=0 \qquad \forall v\in Y,
\end{align}
where existence and uniqueness of $u^\star_Y$ follows again from the main theorem of strongly monotone operators.

Then, 
ILG~\cite{HeidWihler:19v2} is based on restricting the weak iteration scheme~\eqref{eq:itweak} to $Y$. Specifically, for a prescribed initial guess $u_Y^0 \in Y$, define a sequence $\{u_Y^n\}_{n \geq 0}\subset Y$ inductively by
\begin{equation}\label{eq:itweakY}
a(u_Y^n;u_Y^{n+1},w)=\dprod{f(u_Y^n),w}\qquad \forall w\in Y.
\end{equation}
Note that \eqref{eq:itweakY} admits a unique solution, since the conditions~\eqref{eq:coercive} and~\eqref{eq:continuity} above remain valid for the restriction to $Y \subseteq X$. For the purpose of the convergence results in this paper, we require that~\eqref{eq:F=0} originates from an energy minimization problem:
\begin{enumerate}[(F1)]
\setcounter{enumi}{2}
\item The operator $\F$ possesses a \emph{potential}, i.e.,~there exists a G\^{a}teaux differentiable (energy) functional $\H:X \to \mathbb{R}$ such that $\H'=\F$.
\end{enumerate}
Furthermore, we suppose a \emph{monotonicity condition} on the energy functional $\E$ to hold:
\begin{enumerate}[(F1)]
\setcounter{enumi}{3}
\item There exists a (uniform) constant $C_\H>0$ such that, for any closed subspace $Y \subseteq X$,
the sequence defined by~\eqref{eq:itweakY} fulfils the bound
 \begin{align} \label{eq:Hconstant}
  \H(u_Y^{n})-\H(u_Y^{n+1}) \geq C_\H \norm{u_Y^{n}-u_Y^{n+1}}_X^2 \qquad \forall n \geq 0,
 \end{align}
where $\H$ is the potential of $\F$ (restricted to $Y \subseteq X$) introduced in (F3).
\end{enumerate}
Then, it is well-known (see, e.g.,~\cite[Lem.~2]{HeidWihler2:19v1}) that~\eqref{eq:F=0Y} can equivalently be formulated as follows:
\begin{align*}
\text{Find $u^\star_Y \in Y$ such that} \quad \E(u^\star_Y) = \min_{v \in Y} \E(v).
\end{align*}
Moreover, we have the following energy contraction result for ILG~\eqref{eq:itweakY}.

\begin{theorem} \label{prop:energyreduction}
Consider the sequence~$\{u_Y^n\}_{n\ge 0}\subset Y$ generated by the iteration~\eqref{eq:itweakY}. If the conditions~{\rm (F1)--(F4)} are satisfied, and, for any $u \in Y$, the form $a(u;\cdot,\cdot)$ fulfills \eqref{eq:coercive}--\eqref{eq:continuity}, then there holds the energy contraction property
\begin{align}\label{eq:energyreduction0}
0\le \E(u_Y^{n+1})-\E(u_Y^\star) \leq q_{\rm ctr}^2 \left[\E(u_Y^n)-\E(u_Y^\star)\right]\qquad \forall n\ge 0,
\end{align}
with a contraction constant
\begin{align}\label{eq:energyreduction}
 0 \le q_{\rm ctr}:=\left(1-{ 2 C_\H \nu^2}{\beta^{-2} \Flc^{-1}}\right)^{\nicefrac12} < 1
\end{align}
independent of the subspace $Y$ and of the iteration number~$n$. In particular, the sequence~$\{u_Y^n\}_{n\ge 0}$ converges to the unique solution $u_Y^\star \in Y$ of \eqref{eq:F=0Y}.
\end{theorem}

\begin{proof}
With (F2) and since $u_Y^\star$ is the (unique) solution of \eqref{eq:F=0Y}, for~$n\ge 0$, we first observe that
 \begin{align*}
  \Fsm \norm{u_Y^\star-u_Y^{n}}_X^2 &\stackrel{\rm(F2)}\leq \dprod{\F(u_Y^\star)-\F(u_Y^{n}),u_Y^\star-u_Y^{n}} 
  \stackrel{\eqref{eq:F=0Y}}= \dprod{\F(u_Y^{n}),u_Y^{n}-u_Y^\star}.
 \end{align*}
Recalling that $f(u)=\A[u]u-\F(u)$,~\eqref{eq:itweakY} and \eqref{eq:continuity} prove that
\begin{align*}
 \dprod{\F(u_Y^{n}),u_Y^{n}-u_Y^\star}
 \stackrel{\eqref{eq:itweakY}}= a(u_Y^{n};u_Y^{n}-u_Y^{n+1},u_Y^{n}-u_Y^\star)
 \stackrel{\eqref{eq:continuity}}\leq \beta \norm{u_Y^{n+1}-u_Y^{n}}_X \norm{u_Y^{n}-u_Y^\star}_X.
\end{align*}
Altogether, we derive the {\sl a~posteriori} error estimate
\begin{equation} \label{eq:discreteerrorestimate} 
 \norm{u_Y^\star-u_Y^{n}}_X \leq \beta\nu^{-1} \norm{u_Y^{n}-u_Y^{n+1}}_X \qquad \forall n \ge 0.
\end{equation}
Next, we exploit the structural assumptions (F1)--(F3) to obtain the well-known inequalities
\begin{equation} \label{eq:energydifference}
  \frac{\Fsm}{2} \norm{u_Y^\star - v}_X^2 \leq \H(v)-\H(u_Y^\star) \leq \frac{\Flc}{2} \norm{u_Y^\star-v}^2_X \qquad \forall v \in Y;
 \end{equation}
see, e.g., \cite[Lem.~2]{HeidWihler2:19v1} or \cite[Lem.~5.1]{GantnerHaberlPraetoriusStiftner:17}. For any $n\ge 0$, we thus infer that
\begin{align*}
0 \stackrel{\eqref{eq:energydifference}}\le 
\E(u_Y^{n+1})-\E(u_Y^\star) 
&=\E(u_Y^{n})-\E(u_Y^\star)-[\E(u_Y^{n})-\E(u_Y^{n+1})]\\
&\overset{\eqref{eq:Hconstant}}{\leq} \E(u_Y^{n})-\E(u_Y^\star)-C_\H \norm{u_Y^{n}-u_Y^{n+1}}_X^2\\
&\overset{\eqref{eq:discreteerrorestimate}}{\leq} \E(u_Y^{n})-\E(u_Y^\star)-\frac{C_\H \nu^2}{\beta^2} \norm{u_Y^\star-u_Y^n}_X^2\\
&\overset{\eqref{eq:energydifference}}{\leq} \E(u_Y^{n})-\E(u_Y^\star) -\frac{ 2 C_\H \nu^2}{\beta^2 \Flc}\left[\E(u_Y^{n})-\E(u_Y^\star)\right]\\
&= q_{\rm ctr}^2 \, [\E(u_Y^{n})-\E(u_Y^\star)].
\end{align*}
In particular, this proves that $0 \le q_{\rm ctr}^2 < 1$. 
Iterating this inequality, we obtain that 
\begin{align*}
0 \leq \frac{\nu}{2}\norm{u_Y^\star-u_Y^{n}}^2_X \overset{\eqref{eq:energydifference}}{\leq}  \E(u_Y^{n})-\E(u_Y^\star) 
\leq  q_{\rm ctr}^{2n} \, [\E(u_Y^{0})-\E(u_Y^\star)]
\quad \forall n \ge 0.
\end{align*}
Therefore, we conclude that $u_Y^n \to u_Y^\star$ in $Y$ as $n \to \infty$.
\end{proof}

\begin{remark}
If $\F$ satisfies {\rm (F1)--(F3)} and if the ILG bilinear form $a(\cdot;\cdot,\cdot)$ from 
\eqref{eq:itweak} is coercive~\eqref{eq:coercive} with coercivity constant~$\alpha > \nicefrac{\Flc}{2}$, then {\rm (F4)} is fulfilled; see~\cite[Prop.~1]{HeidWihler2:19v1}. Moreover, upon imposing alternative conditions, we may still be able to satisfy~\eqref{eq:Hconstant} even when~$\alpha \le \nicefrac{\Flc}{2}$. For instance,~\cite[Rem.~2.8]{HeidWihler:19v2} proposed an {\sl a~posteriori} step size strategy that guarantees the bound~\eqref{eq:Hconstant} in the context of the damped Newton method. This argument can be generalized to other methods containing a damping parameter. 
\end{remark}

\subsection{Examples}

Let $\F$ from~\eqref{eq:F=0} satisfy~(F1)--(F3). In this section, we briefly recall three examples from~\cite{HeidWihler:19v2,HeidWihler2:19v1} that fulfill~\eqref{eq:coercive}--\eqref{eq:continuity} as well as~(F4), and thus fit into the abstract framework of the previous subsection.

\subsubsection*{Zarantonello (or Picard) linearization approach:} The Zarantonello iteration reads
\begin{align}\label{eq:zarantonello}
 (u_Y^{n+1}, w)_X = (u_Y^{n}, w)_X - \delta \dprod{\F(u_Y^n),w}
 \quad \forall w \in Y~\forall n\ge 0;
\end{align}
cf.~Zarantonello's original report~\cite{Zarantonello:60} or the monographs~\cite[\S3.3]{Necas:86} and~\cite[\S25.4]{Zeidler:90}.

\begin{proposition}[{\hspace*{-3pt}\cite[Thm.~2.2]{HeidWihler:19v2}, \cite[Rem.~1]{HeidWihler2:19v1}}]\label{prop:zarantonello}
The Zarantonello linearization approach~\eqref{eq:zarantonello} fits into the iterative linearization framework with $\A[u]v = \delta^{-1}(v,\cdot)_X$ and $\alpha = \delta^{-1} = \beta$. If $0 < \delta < \nicefrac{2\nu}{L_\F^2}$, then there holds norm contraction 
\begin{align}\label{eq:zarantonello:norm}
 0 \le \norm{u_Y^\star - u_Y^{n+1}}_{X}^2 \le [1 - \delta(2\nu -\delta L_\F^2)] \, \norm{u_Y^\star - u_Y^{n}}_{X}^2
 \qquad \forall n \ge 0.
\end{align}
Finally, any choice $0 < \delta < \nicefrac{2}{L_\F}$ guarantees the validity of {\rm(F4)} with $C_\E = \nicefrac{1}{\delta} - \nicefrac{L_\F}{2} > 0$, and hence the energy contraction~\eqref{eq:energyreduction0}--\eqref{eq:energyreduction}.
\end{proposition}

\subsubsection*{Ka\v{c}anov linearization approach:} Suppose that the nonlinear operator $\F$ from~\eqref{eq:F=0} takes the form $\F(u) = \A[u]u -g$, with $\A[u] \in \mathcal{L}(X,X^\star)$ and $g \in X^\star$. Then, the Ka\v{c}anov iteration reads
\begin{align}\label{eq:kacanov}
 \dprod{\A[u_Y^n]u_Y^{n+1}, w} = \dprod{g,w}
 \qquad \forall w \in Y ~ \forall n\ge 0,
\end{align}
where $\A[\cdot]$ also takes the role of the preconditioning operator for the iterative linearization; cf.~Ka\v{c}anov's work~\cite{KACHANOV1959} introducing the iteration scheme in the context of variational methods for plasticity problems, or the monograph \cite[\S25.13]{Zeidler:90}.

\begin{proposition}[{\hspace*{-3pt}\cite[Thm.~2.5]{HeidWihler:19v2}}]
Suppose that $\A$ is symmetric, i.e.,
\begin{align*}
 \dprod{\A[u] v, w} = \dprod{\A[u] w, v}
 \quad \forall u,v,w \in X,
\end{align*}
and that the energy functional $\E$ satisfies that
\begin{align*}
 \E(u) - \E(v) \ge \frac12 \big[ \dprod{\A[u] u, u} - \dprod{\A[u] v, v} \big]
 \quad \forall u,v \in X.
\end{align*}
Then, the Ka\v{c}anov linearization approach~\eqref{eq:kacanov} guarantees validity of {\rm(F4)} with $C_\E = \nicefrac{\alpha}{2}$, and hence the energy contraction~\eqref{eq:energyreduction0}--\eqref{eq:energyreduction}.
\end{proposition}

\subsubsection*{Damped Newton method:} Suppose that the nonlinear operator $\F$ from~\eqref{eq:F=0} is G\^ateaux differentiable. Then, the damped Newton method reads
\begin{align}\label{eq:newton}
 [\F'(u_Y^n)] u_Y^{n+1} = [\F'(u_Y^n)] u_Y^n - \delta(u_Y^n) \F(u_Y^n)
 \qquad \forall n\ge 0,
\end{align}
where $\delta(u_Y^n) > 0$ is a damping parameter.

\begin{proposition}[{\hspace*{-4pt}\cite[Thm.~2.6]{HeidWihler:19v2}}]\label{prop:newton}
For any~$u\in X$, suppose that the bilinear form $a'(u;v,w) := \dprod{[\F'(u)] v,w}$ is uniformly coercive and bounded, i.e., there are two constants~$\alpha',\beta'>0$ independent of $u \in X$ such that
\begin{equation*}
a'(u;v,v) \geq \aco' \|v\|_X^2 \qquad\forall v \in X,
\end{equation*}
and
\begin{equation*}
a'(u;v,w) \leq \abd' \norm{v}_X \norm{w}_X \qquad\forall  v,w \in X.
\end{equation*}
Let $0 < \delta_{\rm min} < \delta_{\rm max} < \nicefrac{2\alpha'}{L_\F}$. Then, provided that $\delta_{\rm min} \le \delta(u) \le \delta_{\rm max}$, the damped Newton method~\eqref{eq:newton} fits into the iterative linearization framework with $\A[u] = \delta(u)^{-1}\F'(u)$, $\alpha = \nicefrac{\alpha'}{\delta_{\rm max}}$, and  $\beta = \nicefrac{\beta'}{\delta_{\rm min}}$. Moreover, it there holds {\rm (F4)} with $C_\E = \nicefrac{\alpha'}{\delta_{\rm max}} - \nicefrac{L_\F}{2} > 0$, and hence the energy contraction~\eqref{eq:energyreduction0}--\eqref{eq:energyreduction}.
\end{proposition}

%

\section{Adaptive iterative linearized finite element method (AILFEM)}\label{section:ailfem}


In this section, we thoroughly formulate AILFEM. 
Throughout,
we assume that $\F$ satisfies {\rm (F1)--(F4)} and that \eqref{eq:coercive} and \eqref{eq:continuity} hold true. We will apply the ILG approach~\eqref{eq:itweakY} to a sequence of nested Galerkin subspaces $X_0 \subset X_1\subset X_2\subset\ldots \subset X_N\subset\ldots\subset X$, with corresponding sequences $\{u_N^n\}_{n \ge 0} \subset X_N$, for $N \ge 0$, obtained from the weak formulation
\begin{equation} \label{eq:lindisproblemN}
 a(\un{n};\un{n+1},v)=\dprod{f(\un{n}),v} \qquad \forall v \in X_N.
\end{equation}%
For each $N\ge 0$, the space $X_N$ employed in~\eqref{eq:lindisproblemN} will be a conforming finite element space that is associated to an admissible triangulation $\mathcal{T}_N$ of an underlying bounded Lipschitz domain $\Omega \subset \mathbb{R}^d$, with 
$d \ge 2$. 
Then, AILFEM (Algorithm~\ref{alg:praetal} below)
exploits the interplay of adaptive mesh refinements and the iterative scheme~\eqref{eq:lindisproblemN}; for given $u_0^0 \in X_0$, the initial guesses for the iterations on each individual Galerkin space are defined by $u_{N+1}^0 = u_N^{\nN} \in X_N \subset X_{N+1}$ for all $N \ge 0$ and appropriate indices $\nN \ge 1$.

\subsection{Mesh refinements}

We adopt the framework from~\cite[\S2.2--2.4]{GHPS:2020}, with slightly modified notation. 
Consider a shape-regular mesh refinement strategy $\refine(\cdot)$ such as, e.g., the newest vertex bisection~\cite{Mitchell:91}. For a subset $\mathcal{M}_N$ of marked elements of a regular 
triangulation $\mathcal{T}_N$, let $\refine(\mathcal{T}_N,\mathcal{M}_N)$ be the coarsest regular refinement of $\mathcal{T}_N$ such that all elements $\mathcal{M}_N$ have been refined. Specifically, we write $\refine(\mathcal{T}_N)$ for the set of all possible meshes that can be generated from $\mathcal{T}_N$ by (repeated) use of $\refine(\cdot)$. For a mesh $\Tref_N \in \refine(\mathcal{T}_N)$, we assume the nestedness of the corresponding finite element spaces $X^{\rm ref}_N$ and $X_N$, respectively, i.e., $X_N \subseteq X^{\rm ref}_N$. In the sequel, starting from a given initial triangulation $\mathcal{T}_0$ of $\Omega$, we let $\mathbb{T}:=\refine(\mathcal{T}_0)$ be the set of all possible refinements of $\mathcal{T}_0$.

With regards to the optimal convergence rate of the algorithm with respect Section~\ref{sc:convergence}), a few assumptions on the mesh refinement strategy are required, cf.~\cite[\S2.8]{GHPS:2020}. These are satisfied, in particular, for the newest vertex bisection. 

\begin{enumerate}[(R1)]
 \item \emph{Splitting property:} Each refined element is split into at least two and at most $\cson$ many subelements, where $\cson\ge 2 $ is a generic constant. In particular, for all $\mathcal{T} \in \mathbb{T}$, and for any $\mathcal{M} \subseteq \mathcal{T}$, the (one-level) refinement $\Trefo:=\refine(\mathcal{T},\mathcal{M})$  satisfies 
 \begin{align*}
  \# (\mathcal{T} \setminus \Trefo) + \# \mathcal{T} \leq \# \Trefo \leq \cson \# (\mathcal{T} \setminus \Trefo)+\# (\mathcal{T} \cap \Trefo).
 \end{align*}
 Here, $\mathcal{T} \setminus \Trefo$ is the set of all elements in $\mathcal{T}$ which have been refined in $\Trefo$, and $\mathcal{T} \cap \Trefo$ comprises all unrefined elements. 
 \item \emph{Overlay estimate:} For all meshes $\mathcal{T} \in \mathbb{T}$ and all refinements $\Tref_1,\Tref_2 \in \refine(\mathcal{T})$, there exists a common refinement denoted by 
 \[
 \Tref_1 \oplus \Tref_2 \in \refine(\Tref_1) \cap \refine(\Tref_2) \subseteq \refine(\mathcal{T}),
 \] 
 which satisfies 
$  \# (\Tref_1 \oplus \Tref_2) \leq \#\Tref_1 + \# \Tref_2-\# \mathcal{T}$.
\item \emph{Mesh-closure estimate:} There exists a constant $\cmesh>0$ such that, for each sequence $\{\mathcal{T}_{N}\}_{N \geq 1}$ of successively refined meshes, i.e., $\mathcal{T}_{{N+1}}:=\refine(\mathcal{T}_{N},\mathcal{M}_{N})$ for some $\mathcal{M}_{N} \subseteq \mathcal{T}_{N}$, there holds that
\begin{equation*}
 \# \mathcal{T}_{N} -\# \mathcal{T}_0 \leq \cmesh \sum_{J=0}^{N-1} \# \mathcal{M}_{J} \qquad \forall N\in\mathbb{N}.
\end{equation*}
\end{enumerate}

\subsection{Error estimators}

For a mesh $\mathcal{T}_N \in \mathbb{T}$ associated to a discrete space $X_N$, suppose that there exists a \emph{computable local refinement indicator} $\eta_N:\,\mathcal{T}_N\times X_N\to\mathbb{R}$, with $\eta_N(T,v) \geq 0$ for all $T \in \mathcal{T}_N$ and $v \in X_N$. Then, for any $v \in X_N$ and $\mathcal{U}_N \subseteq \mathcal{T}_N$, let 
\begin{align}\label{eq:eta}
\eta_N(\mathcal{U}_N,v):= \bigg( \sum_{T \in \mathcal{U}_N} \eta_N(T,v)^2 \bigg)^{\nicefrac{1}{2}} \qquad \text{and} \qquad \eta_N(v):=\eta_N(\mathcal{T}_N,v).
\end{align}
We recall the following axioms of adaptivity from~\cite{CarstensenFeischlPagePraetorius:14} for the refinement indicators: There are fixed constants $\cstb, \ \crel \geq 1$ and $0<\qred<1$ such that, for all $\mathcal{T}_N \in \mathbb{T}$ and $\Tref_N \in \refine(\mathcal{T}_N)$, with associated refinement indicators $\eta_N$ and $\eta^{\rm ref}_N$, the following properties hold, cf.~\cite[\S2.8]{GHPS:2020}:
\begin{enumerate}[({A}1)]
\item \emph{Stability:} $|\eta_N(\mathcal{U}_N,v) - \eta^{\rm ref}_N(\mathcal{U}_N,w)| \leq \cstb \norm{v-w}_X$, for all $v \in X_N, w \in X^{\rm ref}_N$ and all $\mathcal{U}_N \subseteq \mathcal{T}_N \cap \Tref_N$.
\item \emph{Reduction:} $\eta^{\rm ref}_N(\Tref_N \setminus \mathcal{T}_N,v) \leq \qred \eta_N(\mathcal{T}_N \setminus \Tref_N,v)$, for all $v \in X_N$.
\item \emph{Reliability:}  For the error between the exact solution $u^\star\in X$ of~\eqref{eq:F=0} and the exact discrete solution $u^\star_N \in X_N$ of~\eqref{eq:F=0Y}, we have the {\sl a~posteriori} error estimate $\norm{u^\star-\dissol{N}}_X \leq \crel \eta_N(\dissol{N})$.
\item \emph{Discrete reliability:} 
$\norm{u^{\star,{\rm ref}}_N-\dissol{N}}_X \leq \crel \eta_N (\mathcal{T}_N \setminus \Tref_N,\dissol{N})$, where $u^{\star,{\rm ref}}_N \in X^{\rm ref}_N$ is the solution of~\eqref{eq:F=0Y} on the discrete space $X^{\rm ref}_N$ associated with $\Tref_N$.
\end{enumerate}

We emphasize that (A1)--(A4) are satisfied for the usual $h$-weighted residual error estimators in the specific context of our application in Section~\ref{sec:examples}, cf.~\eqref{eq:errorestimatorscl}. 

\subsection{AILFEM algorithm}

We recall the adaptive algorithm from~\cite{GantnerHaberlPraetoriusStiftner:17} (and its generalization~\cite{GHPS:2020}), which was studied in the specific context of finite element discretizations of the Zarantonello iteration; see Algorithm~\ref{alg:praetal}. These algorithms are closely related to the general adaptive ILG approach in~\cite{HeidWihler:19v2}. The key idea is the same in all algorithms: On a given discrete space, we iterate the linearization scheme~\eqref{eq:lindisproblemN} as long as the linearization error estimator dominates. Once the ratio of the linearization error estimator $[\H(u_N^n) - \H(u_N^{n-1})]^{\nicefrac{1}{2}}$ and the {\sl a~posteriori} error estimator $\eta_N(u_N^n)$ falls below a prescribed tolerance, the discrete space is refined appropriately. 

\begin{algorithm} 
\caption{(AILFEM)}
\label{alg:praetal}
\begin{algorithmic}[1]
\State Prescribe adaptivity parameters $\lambda >0$, $0 < \theta \leq 1$, and $\cmark \geq 1$. Moreover, set~$N:=0$ and~$n:=0$. Start with an initial triangulation $\mathcal{T}_0$, a corresponding finite element space $X_0$, and an arbitrary initial guess $u_0^0 \in X_0$.
\While{true}
\Repeat\ with $n \gets 0$
\State Perform a single iterative linearization step~\eqref{eq:lindisproblemN} on $X_N$ to obtain $\un{n+1}$ from $\un{n}$.
\State Update $n \gets n+1$.
\Until $[\E(\un{n})-\E(\un{n-1})]^{\nicefrac{1}{2}} \le \lambda \eta_N(\un{n})$
\State \multiline{Determine a marking set $\mathcal{M}_N \subseteq \mathcal{T}_N$ with minimal cardinality (up to the multiplicative constant $\cmark\ge 1$) satisfying the D\"orfler criterion $\theta \eta_N(\un{n}) \leq \eta_N(\mathcal{M}_N,\un{n})$ from~\cite{doerfler1996}, and set $\mathcal{T}_{N+1}:=\refine(\mathcal{T}_N,\mathcal{M}_N)$.}
\State Set $\n(N):= n$ and define $u_{N+1}^0 := \un{\nN}$ by inclusion $X_{N+1} \supseteq X_N$.
\State Update $N\gets N+1$.
\EndWhile.
\end{algorithmic}
\end{algorithm}

\begin{remark}\label{rem:algorithm}
Under the conditions {\rm (A1)} and {\rm (A3)}, we notice some facts about Algorithm~\ref{alg:praetal} from \cite[Prop.~3]{GHPS:2020} and \cite[Prop.~4.4 \& 4.5]{GantnerHaberlPraetoriusStiftner:17}: For any $N \geq 0$, there holds the {\sl a~posteriori} error estimate
\begin{align}\label{eq:aposterioriestimator}
\norm{u^\star-\un{\nN}}_X \leq C_{\rm rel}' \eta_N(\un{\nN}),
\end{align}
where $C_{\rm rel}'$ depends only on $\nu,\beta,\lambda,\Flc,\crel,\cstb$, and $C_\E$. In particular, if the repeat loop terminates with $\eta_N(\un{\nN})=0$ for some $N \geq 0$, then $\un{\nN}=u^\star$, i.e.,~the exact solution is obtained. Moreover, in the non-generic case that the repeat loop in Algorithm~\ref{alg:praetal} does not terminate after finitely many steps, for some $N \geq 0$, the generated sequence $\{u_N^n\}_{n \geq 0}$ converges to $\un{\star}=u^\star$ (in particular, the solution $u^\star$ is discrete).
\end{remark}

\section{Optimal convergence of AILFEM}\label{sc:convergence}

We are now ready to outline the linear convergence of the sequence of approximations from Algorithm~\ref{alg:praetal} to the unique solution of \eqref{eq:F=0}. In addition, the rate optimality with respect to the overall computational costs will be discussed. 

\subsection{Step counting}

Following~\cite{GHPS:2020}, we introduce an ordered index set 
\[
\mathcal{Q}:=\{(N,n) \in \mathbb{N}_0^2: \text{ index pair } (N,n) \text{ occurs in Algorithm~\ref{alg:praetal} $\wedge$ $n < \nN$}\},
\]
where $\nN\ge 1$ counts the number of steps in the repeat loop for each $N$. We exclude the pair $(N,\nN)$ from $\mathcal{Q}$, since either $u_{N+1}^0=\un{\nN}$ and $(N+1,0) \in \mathcal{Q}$ or even $\nN := \infty$ if the loop does not terminate after finitely many steps; see also Remark~\ref{rem:algorithm}. Observing that Algorithm~\ref{alg:praetal} is sequential, the index set $\mathcal{Q}$ is naturally ordered: For $(N,n),(N',n') \in \mathcal{Q}$, we write  $(N',n')<(N,n)$ if and only if  $(N',n')$ appears earlier in Algorithm~\ref{alg:praetal} than $(N,n)$.  With this order, we can define the \emph{total step counter}
\[
|(N,n)|:=\#\{(N',n') \in \mathcal{Q}:(N',n')<(N,n)\}=n+\sum_{N'=0}^{N-1} \n(N'),
\]
which provides the total number of solver steps up to the computation of $\un{n}$. Finally, we introduce the notation $\uN:=\sup\{N \in \mathbb{N}:(N,0) \in \mathcal{Q}\}$. 

\subsection{Linear convergence}

Based on the notation used in Algorithm~\ref{alg:praetal}, let us introduce the quasi-error
\begin{align} \label{eq:quasierror}
\Delta_N^n:=\norm{u^\star-\un{n}}_X+\eta_N(\un{n}) \qquad \forall (N,n) \in \overline{\mathcal{Q}}:=\mathcal{Q} \cup \{(N,\nN):\nN < \infty\},
\end{align}
and suppose that the estimator $\eta_N$ satisfies {\rm (A1)--(A3)}. Then, appealing to 
Theorem~\ref{prop:energyreduction}, we make the crucial observation that 
the energy contraction property~\eqref{eq:energyreduction0}--\eqref{eq:energyreduction} of
the ILG method~\eqref{eq:itweakY} 
coincides with the property~\cite[(C1)]{GHPS:2020}.
Consequently,~\cite[Thm.~4]{GHPS:2020} directly applies to our setting:

\begin{theorem}[{\hspace{1sp}\cite[Thm.~4]{GHPS:2020}}]\label{thm:ghps:conv}
Suppose {\rm (A1)--(A3)}. Then, for all $0<\theta\leq 1$ and $0<\lambda<\infty$, there exist constants $\clin \geq 1$ and $0<\qlin<1$ such that the quasi-error~\eqref{eq:quasierror} is linearly convergent in the sense of 
\[
\Delta_N^n \leq \clin \qlin^{|(N,n)|-|(N',n')|}\Delta_{N'}^{n'} \qquad \forall (N,n),(N',n') \in \mathcal{Q} \text{ with } (N',n')<(N,n).
\]
The constants $\clin$ and $\qlin$ depend only on $\Flc,\nu,\beta, C_\H,\cstb,\qred,\crel$, and the adaptivity parameters $\theta$ and $\lambda$.
\end{theorem} 

\begin{remark}
When relating the current results to those of~\cite{GantnerHaberlPraetoriusStiftner:17,GHPS:2020,HeidWihler2:19v1} we note the following observations:
\begin{enumerate}[(a)]
\item The paper~\cite{GHPS:2020} verifies~\cite[(C1)]{GHPS:2020} only for an iterative PCG solver for symmetric, linear, and elliptic PDEs. In addition, for the Zarantonello linearization approach, \cite{GHPS:2020} exploits the weaker contraction property~\cite[(C2)]{GHPS:2020} stemming from~\eqref{eq:zarantonello:norm}. Therefore,~\cite[Thm.~4]{GHPS:2020} merely holds for AILFEM based on the Zarantonello linearization approach under the additional assumption that $0 < \lambda \ll 1$ is sufficiently small. The same restriction on $\lambda$ applies to the more general AILFEM analysis in~\cite{HeidWihler2:19v1}.
\item In \cite{GantnerHaberlPraetoriusStiftner:17,HeidWihler2:19v1}, convergence was proved for the final iterates only, i.e.,
\begin{align*}
 \Delta_N^{\n(N)} \leq \clin \qlin^{N-N'}\Delta_{N'}^{\n(N')} \qquad \forall N, N' \in \N_0
 \text{ with } N' < N \text{ and } (N+1,0) \in \mathcal{Q}.
\end{align*}
\end{enumerate}
Based on Theorem~\ref{prop:energyreduction}, the current analysis constitutes a massive improvement of the above results with respect to the ensuing aspects:
\begin{enumerate}[(a)]
\item Theorem~\ref{thm:ghps:conv} applies to an entire class of linearization approaches for strongly monotone and Lipschitz continuous nonlinearities (including, e.g., the Zarantonello/Picard, Ka\v{c}anov, and damped Newton schemes).
\item Unlike~\cite{GHPS:2020} (and~\cite{GantnerHaberlPraetoriusStiftner:17,HeidWihler2:19v1}), the present setting does no longer require any restriction on $\lambda > 0$ for guaranteed linear convergence of AILFEM. Furthermore, all iterates $\Delta_N^n$ with $(N,n) \in \mathcal{Q}$ are now covered by Theorem~\ref{thm:ghps:conv}.
\end{enumerate} 
\end{remark}%

\subsection{Optimal convergence rate and computational work}

Furthermore, we address the optimal convergence rate of the quasi-error~\eqref{eq:quasierror} with respect to the degrees of freedom as well as the computational work. As before, we can directly apply a result from~\cite{GHPS:2020} owing to the energy contraction from Theorem~\ref{prop:energyreduction}. For its statement, we need further notation: First, for $L \in \mathbb{N}_0$, let $\mathbb{T}(L)$ be the set of all refinements $\mathcal{T}$ of $\mathcal{T}_0$ with $\#\mathcal{T}-\#\mathcal{T}_0 \leq L$. Next, for $s>0$, define 
\begin{align*} 
\norm{u^\star}_{\mathbb{A}_s}:=\sup_{L \in \mathbb{N}}(L+1)^s \inf_{\mathcal{T}_{\rm{opt}} \in \mathbb{T}(L)}\big[\norm{u^\star-\dissol{\rm{opt}}}_X+\eta_{\rm{opt}}(\dissol{\rm{opt}})\big] \in \mathbb{R}_\geq 0 \cup \{\infty\},
\end{align*}
where $\dissol{\rm{opt}}$ is the discrete solution~\eqref{eq:F=0Y} on the finite element space related to an \emph{optimal} (in terms of the above infimum) mesh~$\mathcal{T}_{\rm opt}$. For $s>0$, we note that $\norm{u^\star}_{\mathbb{A}_s}<\infty$ if and only if the quasi-error converges at least with rate $s$ along a sequence of optimal meshes.

\begin{theorem}[{\hspace{1sp}\cite[Thm.~7]{GHPS:2020}}]\label{thm:ghps:optimal}
Suppose {\rm (R1)--(R3)} and {\rm (A1)--(A4)}, and define 
\[
\lambda_{\rm{opt}}:=\frac{1-q_{\rm ctr}}{q_{\rm ctr} \cstb} \sqrt{\nicefrac{\nu}{2}}.
\]
Let $0<\theta\leq 1$ and $0<\lambda<\lambda_{\rm{opt}} \theta$ such that 
\[0<\theta':=\frac{\theta+\nicefrac{\lambda}{\lambda_{\rm{opt}}}}{1-\nicefrac{\lambda}{\lambda_{\rm{opt}}}}<(1+\cstb^2\crel^2)^{-\nicefrac12}.\]
Then, for any $s>0$, there exist positive constants $\copt,\Copt$ such that
\begin{equation} \label{eq:convergencework}
\begin{split}
\copt^{-1} \norm{u^\star}_{\mathbb{A}_s} & \leq \sup_{(N',n') \in \mathcal{Q}} \left(\#\mathcal{T}_{N'}-\#\mathcal{T}_0+1\right)^s \Delta_{N'}^{n'} \\
& \leq \sup_{(N',n') \in \mathcal{Q}} \Bigg(\sum_{\substack{ (N,n) \in \mathcal{Q} \\ (N,n) \leq (N',n')}} \# \mathcal{T}_N\Bigg)^s \Delta_{N'}^{n'}
 \leq \Copt \max\{\norm{u^\star}_{\mathbb{A}_s},\Delta_0^0\}.
\end{split}
\end{equation}
The constant $\copt>0$ depends only on $\nu,\Flc,\cson,\cstb,\crel, \# \mathcal{T}_0$, and $s$, and additionally on $\uN$ and $N_0$, respectively, if $\uN<\infty$ or $\eta_{N_0}(u_{N_0}^{\n(N_0)})=0$ for some $(N_0+1,0) \in \mathcal{Q}$; moreover, the constant $\Copt>0$ depends only on $C_{\rm rel}',\nu,\cstb,\qred,\crel,\cmesh,1-\nicefrac{\lambda}{\lambda_{\rm{opt}}},\cmark,\clin,\qlin,\# \mathcal{T}_0$, and $s$.
\end{theorem}

\begin{remark}
We add a few important comments on the above result.
\begin{enumerate}[(a)]
\item The significance of~\eqref{eq:convergencework} is that the quasi-error $\Delta_N^n$ from~\eqref{eq:quasierror} decays at rate $s$ (with respect to the number of elements or, equivalently, the number of degrees of freedom) if and only if rate $s$ is achievable for the discrete solutions on optimal meshes (with respect to the number of elements). If, in addition, all of the (single) steps in Algorithm~\ref{alg:praetal} can be performed at linear cost, $\mathcal{O}(\# \mathcal{T}_N)$, then the quasi-error even decays with rate $s$ with respect to the overall computational cost if and only if rate $s$ is attainable with respect to the number of elements. Since the total computational cost is proportional to the total computational time, it is therefore monitored in the subsequent numerical experiments.

\item While linear cost is, in practice, a feasible assumption for mesh-refinement, computation of the error estimator, and marking (see, e.g.,~\cite{stevenson2007,pp2020} for D\"orfler marking in linear complexity), we remark that the present result assumes that the arising linear systems are solved exactly in $\mathcal{O}(\# \mathcal{T}_N)$ operations, which is indeed reasonable for state-of-the-art solvers for sparse FEM matrices. 

\item At the price of considering only the Zarantonello linearization approach, the recent work~\cite{hpsv2021}  formulates and analyzes a full AILFEM algorithm, where also the linearized equations are solved approximately by an optimally preconditioned CG method. Overall, the AILFEM algorithm then consists of three nested loops and a triple index set $\mathcal{Q} \subseteq \N_0^3$ for mesh-refinement, Zarantonello linearization, and PCG solver steps. Theorems~\ref{thm:ghps:conv} and~\ref{thm:ghps:optimal} hold accordingly with an additional parameter $\lambda_{PCG} > 0$ for the innermost PCG solver loop, however, the analysis of~\cite{hpsv2021} requires that $0 < \lambda + \lambda_{PCG} \ll 1$ for linear convergence, and $0 < \theta + \lambda + \lambda_{PCG} \ll 1$ for optimal cost. We conjecture that, for linear convergence, it might be sufficient to have $0 < \lambda_{PCG} \ll 1$, while $\lambda > 0$ can now be arbitrary.
%
\end{enumerate}
\end{remark}
%

Finally, the following remark relates the current results to those of~\cite{GantnerHaberlPraetoriusStiftner:17,GHPS:2020,HeidWihler2:19v1}.

\begin{remark}
As for Theorem~\ref{thm:ghps:conv}, we note that~\cite{HeidWihler2:19v1} (following~\cite{GantnerHaberlPraetoriusStiftner:17}) proves only the implication
\begin{align*}
 \norm{u^\star}_{\mathbb{A}_s} < \infty
 \quad \Longrightarrow \quad
 \sup_{(N'+1,0) \in \mathcal{Q}} \Bigg(\sum_{\substack{ (N,n) \in \mathcal{Q} \\ (N,n) \leq (N',n')}} \# \mathcal{T}_N \Bigg)^{s-\varepsilon} \Delta_{N'}^{\n(N')} < \infty,
\end{align*}
for all $0 < \varepsilon < s$. Note that this statement is considerably weaker than that of Theorem~\ref{thm:ghps:optimal}, where $\varepsilon = 0$, and both relations are essentially equivalent.
\end{remark}%

\section{Numerical experiment} \label{sec:examples}

In this section, we test Algorithm~\ref{alg:praetal} with a numerical example.

\subsection{Model problem}

On an open, bounded, and polygonal domain $\Omega \subset \mathbb{R}^2$ with Lipschitz boundary $\Gamma=\partial \Omega$, we consider the quasi-linear second-order elliptic boundary value problem:
\begin{align} \label{eq:operatorscl}
\text{Find } u \in H^1_0(\Omega) \text{ such that} \quad \F(u):= - \nabla \cdot \big\{\mu(\left|\nabla u\right|^2) \nabla{u}\big\}-g=0\quad\text{in }H^{-1}(\Omega),
\end{align}
i.e., with $\nl(\nabla u)=\mu(|\nabla u|^2)\nabla u$ in~\eqref{eq:strongform}.
For $u,v\in X := H^1_0(\Omega)$, the inner product and norm on~$X$ are defined by~$(u,v)_X:=(\nabla u,\nabla v)_{L^2(\Omega)}$ and~$\norm{u}_X:=\|\nabla u\|_{L^2(\Omega)}$, respectively. Let $g \in L^2(\Omega)$ in~\eqref{eq:operatorscl}, embedded as an element in $X^\star=H^{-1}(\Omega)$. Moreover, suppose that the diffusion coefficient $\mu \in C^1([0,\infty))$ fulfills the monotonicity property
\begin{align} \label{en:assmu}
m_\mu(t-s) \leq \mu(t^2)t-\mu(s^2)s \leq M_\mu (t-s) \qquad \forall t \geq s \geq 0,
\end{align}
with constants $M_\mu\ge m_\mu>0$. Under this condition, the nonlinear operator~$\F:\,H^1_0(\Omega)\to H^{-1}(\Omega)$ from~\eqref{eq:operatorscl} satisfies~(F1) and~(F2) with $\nu=m_{\mu}$ and $\Flc=3M_{\mu}$;
see~\cite[Prop.~25.26]{Zeidler:90}. Moreover, $\F$ has a potential $\E:H^1_0(\Omega) \to  \mathbb{R}$ given by
\begin{align*}
\E(u):=\int_\Omega \psi(\left|\nabla u\right|^2) \dx-\dprod{g,u}\qquad \forall u\in H^1_0(\Omega),
\end{align*}
where~$\psi(s):=\nicefrac12\int_0^s \mu(t) \dt$, $s\ge 0$. 
The weak form of the boundary value problem~\eqref{eq:operatorscl} 
reads:
\begin{align}\label{eq:sclweak}
\text{Find } u \in H^1_0(\Omega) \text{ such that} \quad \int_\Omega \mu(|\nabla u|^2) \nabla u \cdot \nabla v \dx = \dprod{g,v} \quad \forall v \in H^1_0(\Omega).
\end{align}

For the nonlinear boundary value problem~\eqref{eq:operatorscl}, 
Propositions~\ref{prop:zarantonello}--\ref{prop:newton} apply and yield the contraction property~\eqref{eq:energyreduction0}--\eqref{eq:energyreduction} for the following linearization schemes:
\begin{enumerate}[(i)]
\item \emph{Zarantonello (or Picard) iteration,} for $\delta_Z \in (0,\nicefrac{2}{(3 M_\mu)})$:
\[
-\Delta u^{n+1}=-\Delta u^{n}-\dpa_Z\F(u^n)\qquad \forall n\ge 0.
\]
\item \emph{Ka\v{c}anov iteration:}
\begin{equation*}
- \nabla \cdot \big\{\mu(\left|\nabla u^n\right|^2) \nabla{u^{n+1}}\big\}-g=0\qquad \forall n\ge 0.
\end{equation*}
\item \emph{Newton iteration,} for a damping parameter $0 < \delta_{\mathrm{min}} \leq \delta(u^n) \leq \delta_{\mathrm{max}}<\nicefrac{2 m_\mu}{3 M_\mu}$:
\[
 \F'(u^{n})u^{n+1}=\F'(u^{n})u^{n}-\dpa(u^n) \F(u^{n})\qquad \forall n\ge 0;
\]
Here, for $u\in X$, the G\^{a}teaux derivative $\F'(u)$ of~$\F$ is given through
\[
\dprod{\F'(u)v,w}=\int_{\Omega} 2 \mu'(|\nabla u|^2)(\nabla u \cdot \nabla v)(\nabla u \cdot \nabla w) \dx + \int_{\Omega} \mu(|\nabla u|^2)\nabla v \cdot \nabla w \dx \qquad \forall v,w\in X. 
\]
\end{enumerate}

\subsection{Discretization and local refinement indicator}{}

AILFEM for~\eqref{eq:sclweak} is based on regular triangulations $\{\mathcal{T}_N\}_{N\ge 0}$ that partition the domain~$\Omega$ into open and disjoint triangles~$T \in\mathcal{T}_N$.
We consider the FEM spaces $X_N:=\left\{v \in H^1_0(\Omega): v|_T \in \mathcal{P}_1(T) \ \forall T \in \mathcal{T}_N\right\}$, where we signify by $\mathcal{P}_1(T)$ the space of all affine functions on $T \in \mathcal{T}_N$. The mesh refinement strategy $\refine(\cdot)$ in Algorithm~\ref{alg:praetal} is given by newest vertex bisection~\cite{Mitchell:91}. Moreover, for any $v \in X_N$ and any $T\in\mathcal{T}_N$, we define the local refinement indicator, respectively the global error indicator from~\eqref{eq:eta}, by 
\begin{align} \label{eq:errorestimatorscl}
\begin{split}
 \eta_N(T,v)^2&:=h_T^2 \twon{g}{T}^2+h_T \twon{\jmp{\mu(|\nabla v|^2)\nabla v}}{\partial T \setminus \Gamma}^2, \\ 
 \eta_N(v)&:=\bigg(\sum_{T \in \mathcal{T}_N}\eta_N(T,v)^2 \bigg)^{\nicefrac{1}{2}},
 \end{split}
\end{align}
where $\jmp{\cdot}$ is the normal jump across element faces, and $h_T := |T|^{\nicefrac12}$ is equivalent to the diameter of $T\in\mathcal{T}$. This error estimator satisfies the assumptions (A1)--(A4) for the problem under consideration; see, e.g.,\cite[\S 3.2]{GarauMorinZuppa:2012} or~\cite[\S 10.1]{CarstensenFeischlPagePraetorius:14}.

\subsection{Computational example}

Consider the L-shaped domain $\Omega=(-1,1)^2 \setminus ([0,1] \times [-1,0])$, and the nonlinear diffusion parameter $\mu(t)=1+\mathrm{e}^{-t}$, which satisfies~\eqref{en:assmu} with $m_\mu=1-2\exp(-\nicefrac{3}{2})$ and $M_\mu=2$. Moreover, we choose $g$ such that the analytical solution of \eqref{eq:operatorscl} is given by 
\[
u^\star(r,\varphi)=r^{\nicefrac{2}{3}}\sin\left(\nicefrac{2\varphi}{3}\right)(1-r \cos\varphi)(1+r \cos\varphi)(1- r \sin\varphi)(1+r \sin\varphi)\cos\varphi,
\]
where $r$ and $\varphi$ are polar coordinates; this is the prototype singularity for (linear) second-order elliptic problems with homogeneous Dirichlet boundary conditions in the L-shaped domain; in particular, we note that the gradient of~$u^\star$ is unbounded at the origin. 

In all our experiments below, we set the adaptive mesh refinement parameters to $\theta=0.5$, and $\cmark=1$. The computations employ an initial mesh $\mathcal{T}_0$ consisting of 192 uniform triangles and the starting guess $u_0^0 \equiv 0$. Then, the procedure is run until the number of elements exceeds $10^6$. We always choose the damping parameter $\delta=1$ for the Newton iteration, and vary the damping parameter $\delta_{Z}$ for the Zarantonello iteration, as well as the adaptivity parameter $\lambda$, cf. line 6 in Algorithm~\ref{alg:praetal}, throughout the experiments. Our implementation is based on the {\sc Matlab} package~\cite{FunkenPraetoriusWissgott:11} with the necessary modifications.

In general, for the Newton scheme, we note that choosing the damping parameter to be~$\dpa=1$ (potentially resulting in quadratic convergence of the iterative linearization close to the solution) is in discord with the assumption in (iii) above, and thus might lead to a divergent iteration for the given boundary value problem (cf.~\cite{AmreinWihler:14}). Our numerical computations illustrate, however, that this is not of concern in the current experiments. Indeed, for $\dpa=1$, the bound~\eqref{eq:Hconstant} from~(F4) remains satisfied in each iteration. Otherwise, a prediction and correction strategy which obeys the bound~\eqref{eq:Hconstant}, could be employed (see~\cite[Rem.~2.8]{HeidWihler:19v2}). This would guarantee the convergence of the (damped) Newton method.

\begin{enumerate}[(1)]
\item $\delta_Z=0.1$ and $\lambda=0.5$: In Figure~\ref{fig:NSC203}, we display the performance of Algorithm~\ref{alg:praetal} with respect to both the number of elements and the measured computational time. We clearly see a convergence rate of $-\nicefrac12$ for the Ka\v{c}anov and Newton method, which is optimal for linear finite elements. Moreover, the Zarantonello iteration has a pre-asymptotic phase of reduced convergence, which becomes optimal for finer meshes. In Figure~\ref{fig:NSCF203} (left) we observe that the energy contraction factor given by 
\begin{equation}\label{eq:EC}
\varkappa_N:=\frac{\E(u_N^{\n(N)})-\E(u^\star)}{\E(u_N^{0})-\E(u^\star)}
\end{equation}
is inferior for the Zarantonello iteration in the initial phase (compared to the Ka\v{c}anov and Newton methods), which might explain the reduced convergence. This contraction factor becomes better for an increased number of iterations, see Figure~\ref{fig:NSCF203} (right), thereby leading to the asymptotically optimal convergence rate for the Zarantonello iteration. Finally, in Figure~\ref{fig:quotient} (left), we display the quotient 
\begin{equation} \label{eq:quotient}
\kappa_N:=\frac{\H(u_N^{\n(N)-1})-\E(u_N^{\n(N)})}{\big\|u_N^{\n(N)-1}-u_N^{\n(N)}\big\|^2_X},
\end{equation} 
which experimentally verifies the assumption (F4).

\begin{figure} 
 \subfloat{\includegraphics[width=0.49\textwidth]{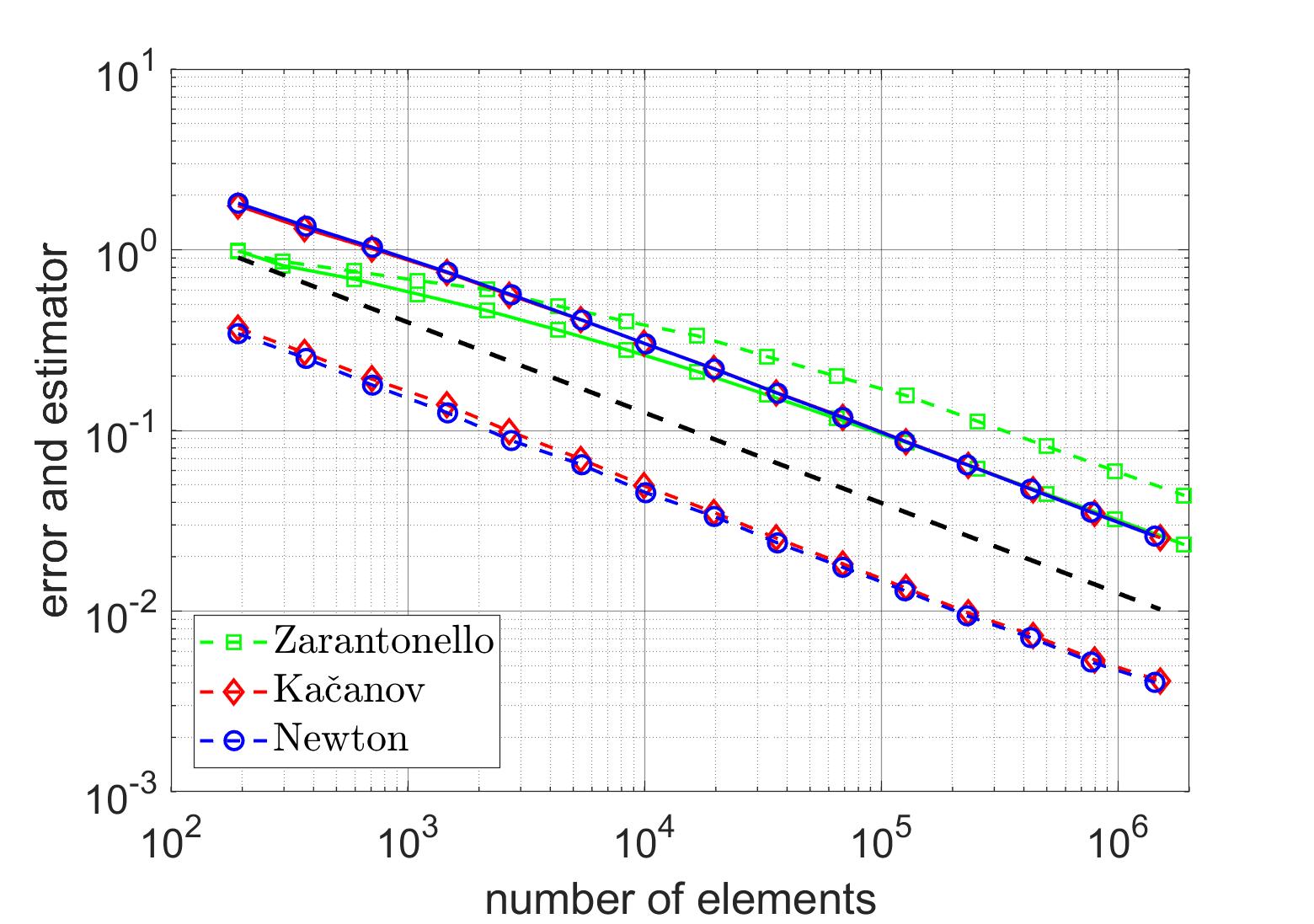}}\hfill
 \subfloat{\includegraphics[width=0.49\textwidth]{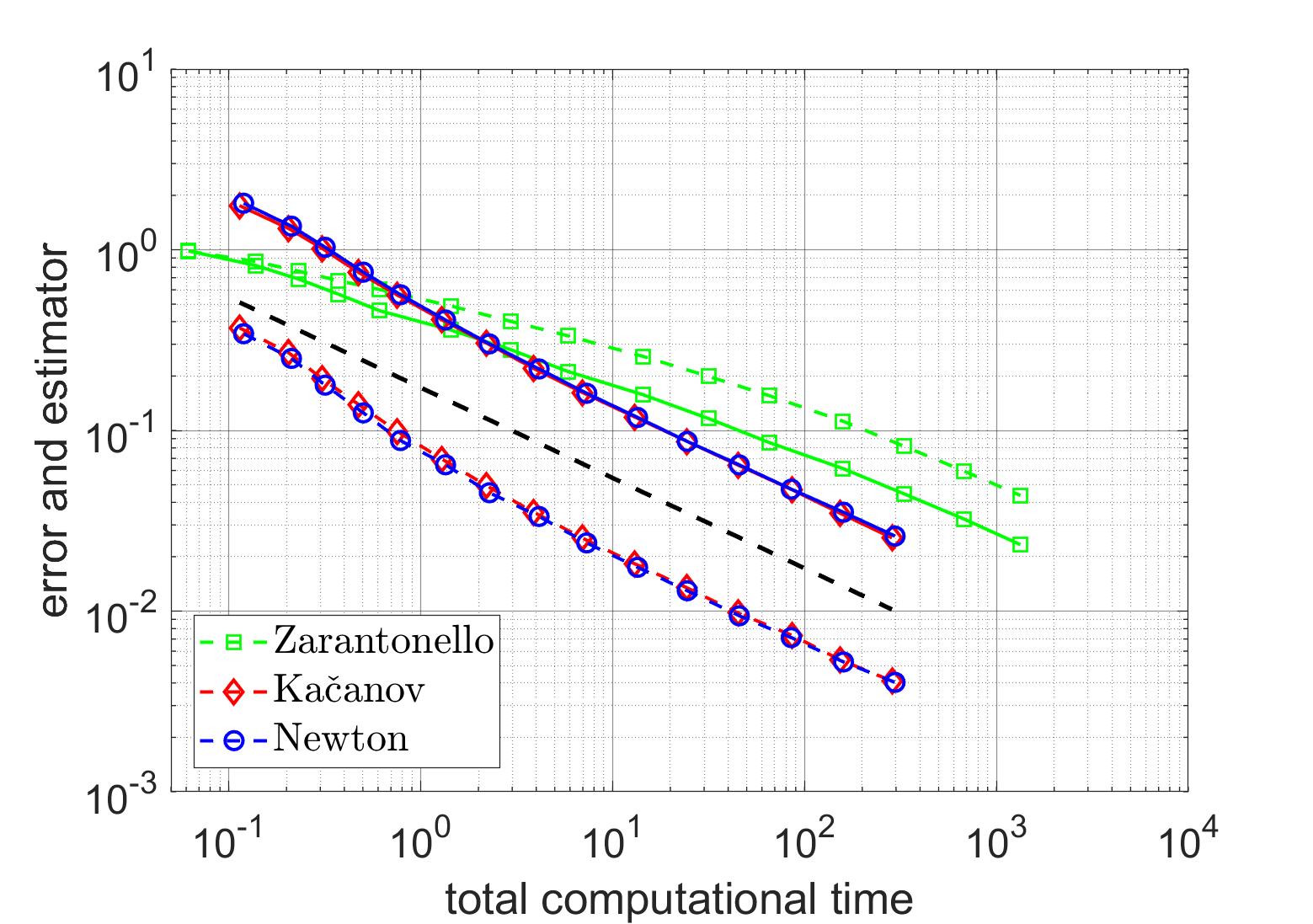}}
 \caption{$\delta_Z=0.1$ and $\lambda=0.5$: Performance plot for adaptively refined meshes with respect to the number of elements (left) and the total computational time (right). The solid and dashed lines correspond to the estimator and the error, respectively. The dashed lines without any markers indicate the optimal convergence order of $-\nicefrac12$ for linear finite elements.} \label{fig:NSC203}
\end{figure}

\begin{figure} 
 \subfloat{\includegraphics[width=0.49\textwidth]{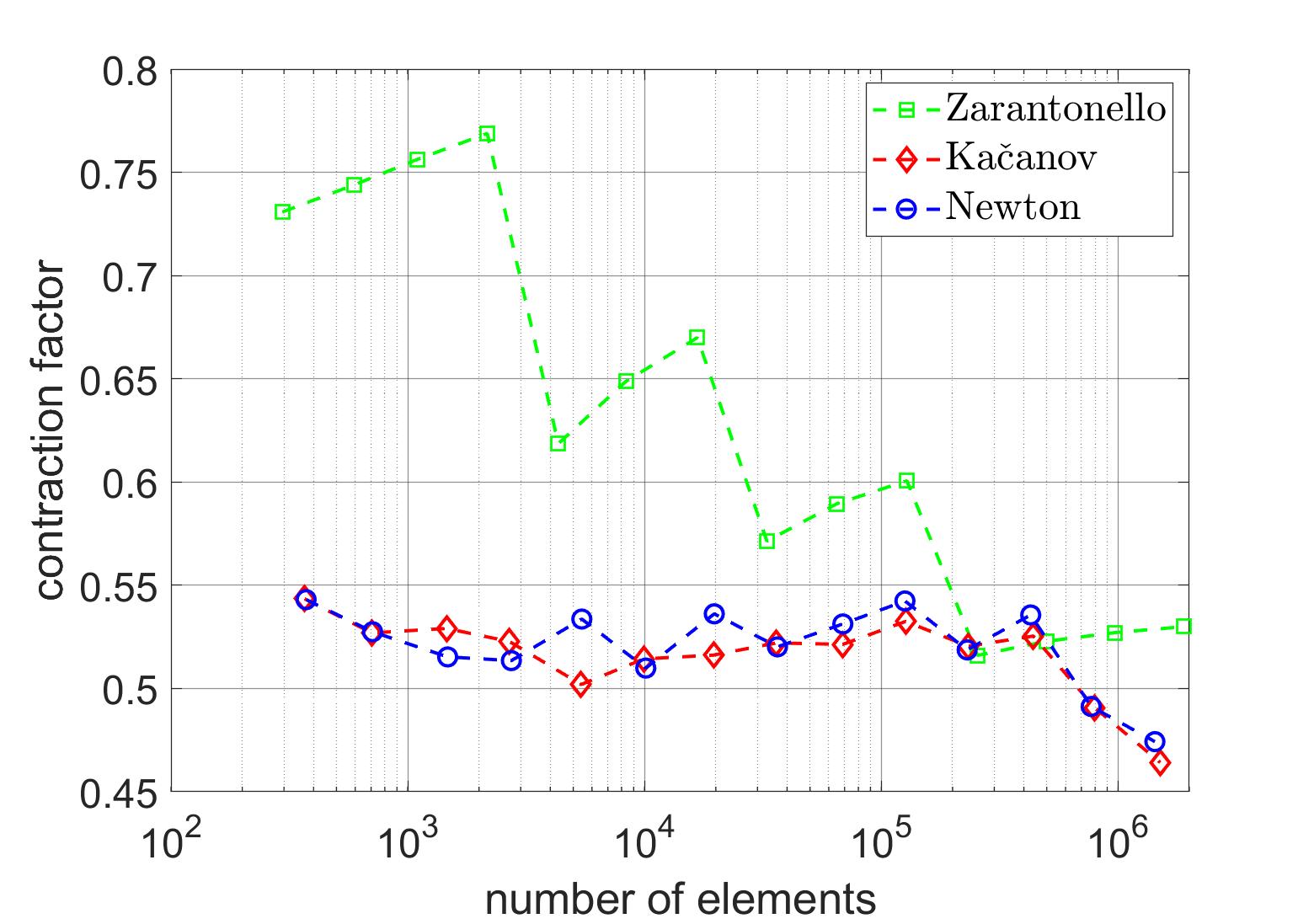}}\hfill
 \subfloat{\includegraphics[width=0.49\textwidth]{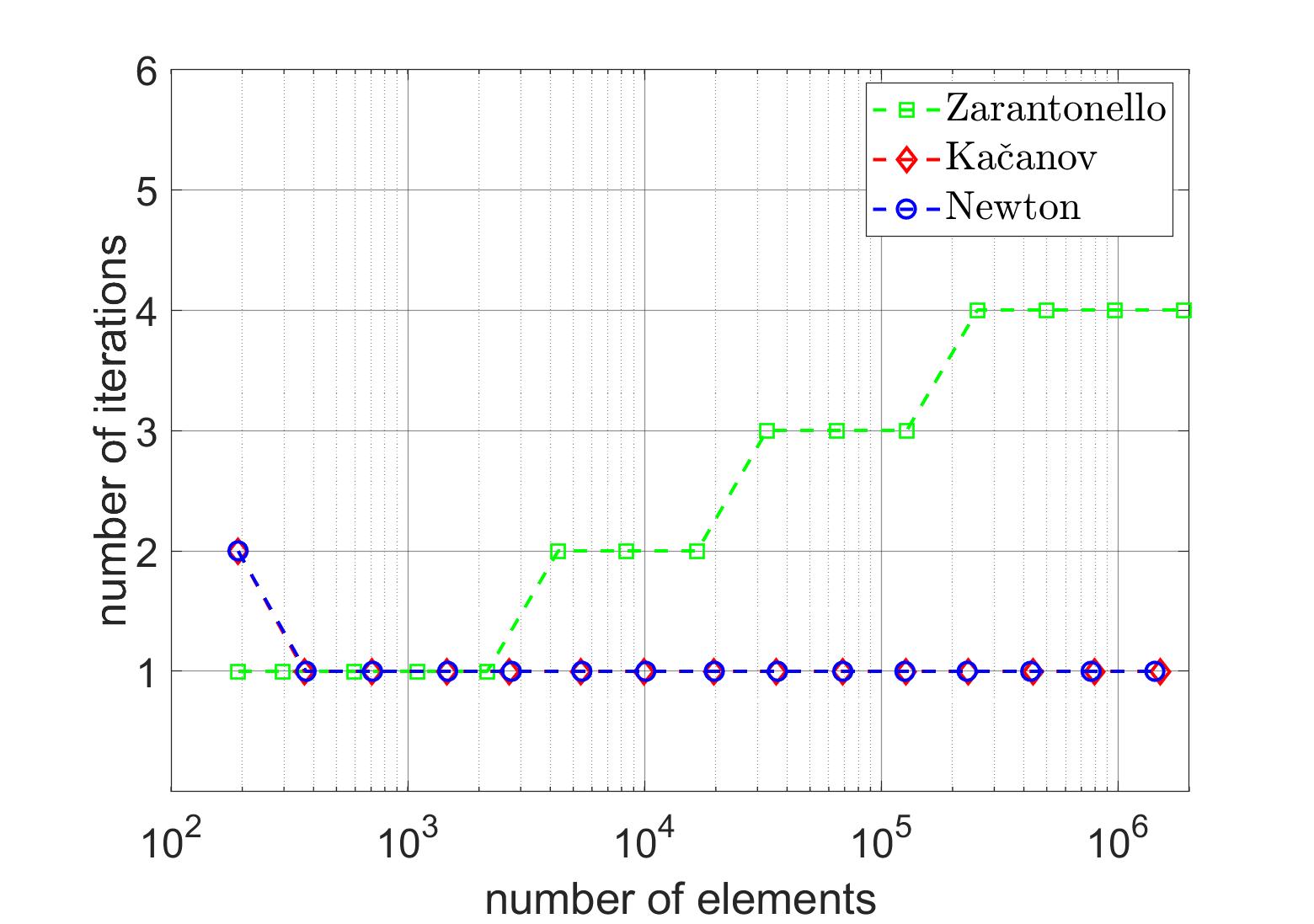}}
 \caption{$\delta_Z=0.1$ and $\lambda=0.5$: The contraction factor $\varkappa_N$ (left, see~\eqref{eq:EC}) and the number of iterations (right) on each finite element space.} \label{fig:NSCF203}
\end{figure}

\begin{figure} 
 \subfloat{\includegraphics[width=0.49\textwidth]{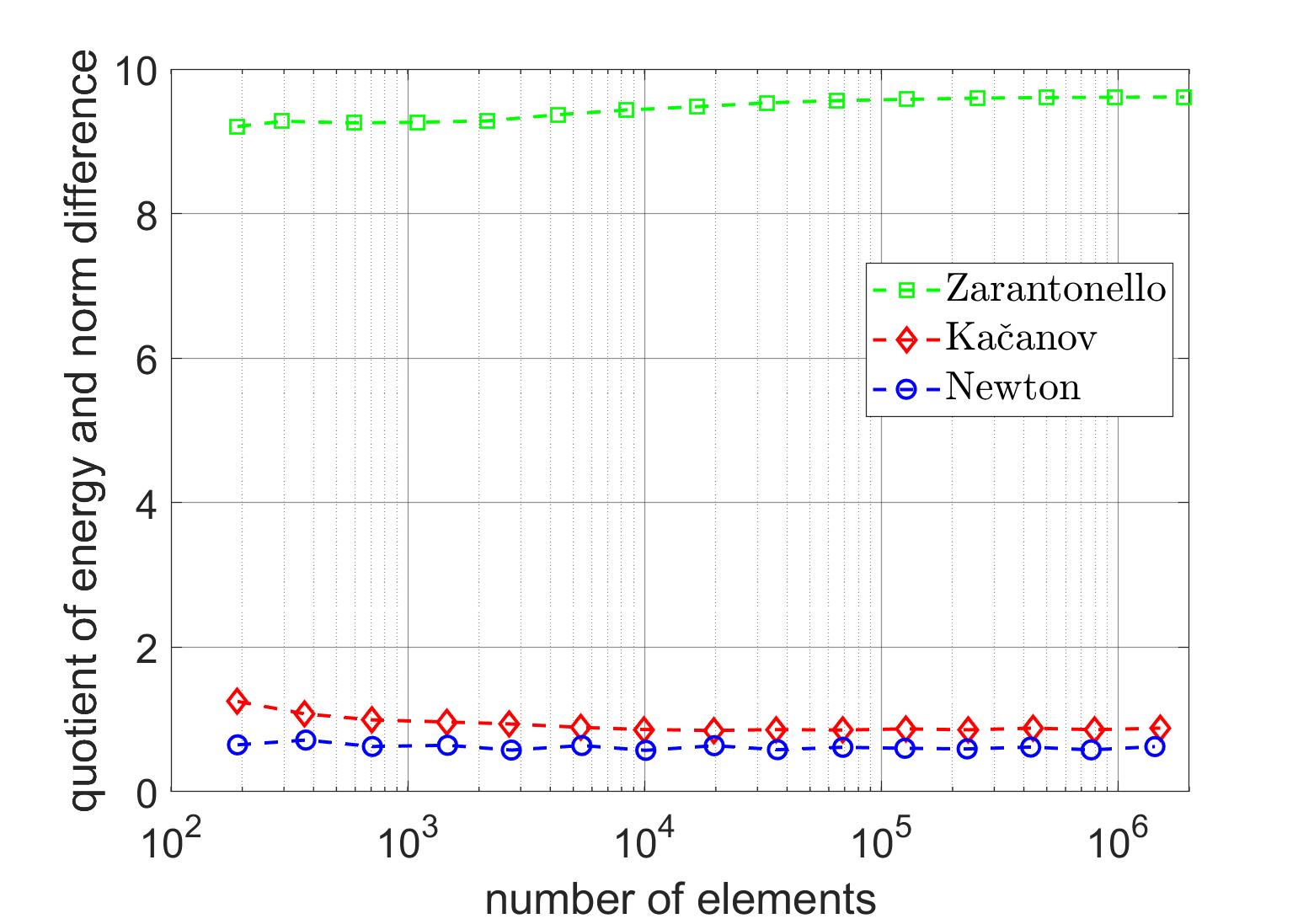}}\hfill
 \subfloat{\includegraphics[width=0.49\textwidth]{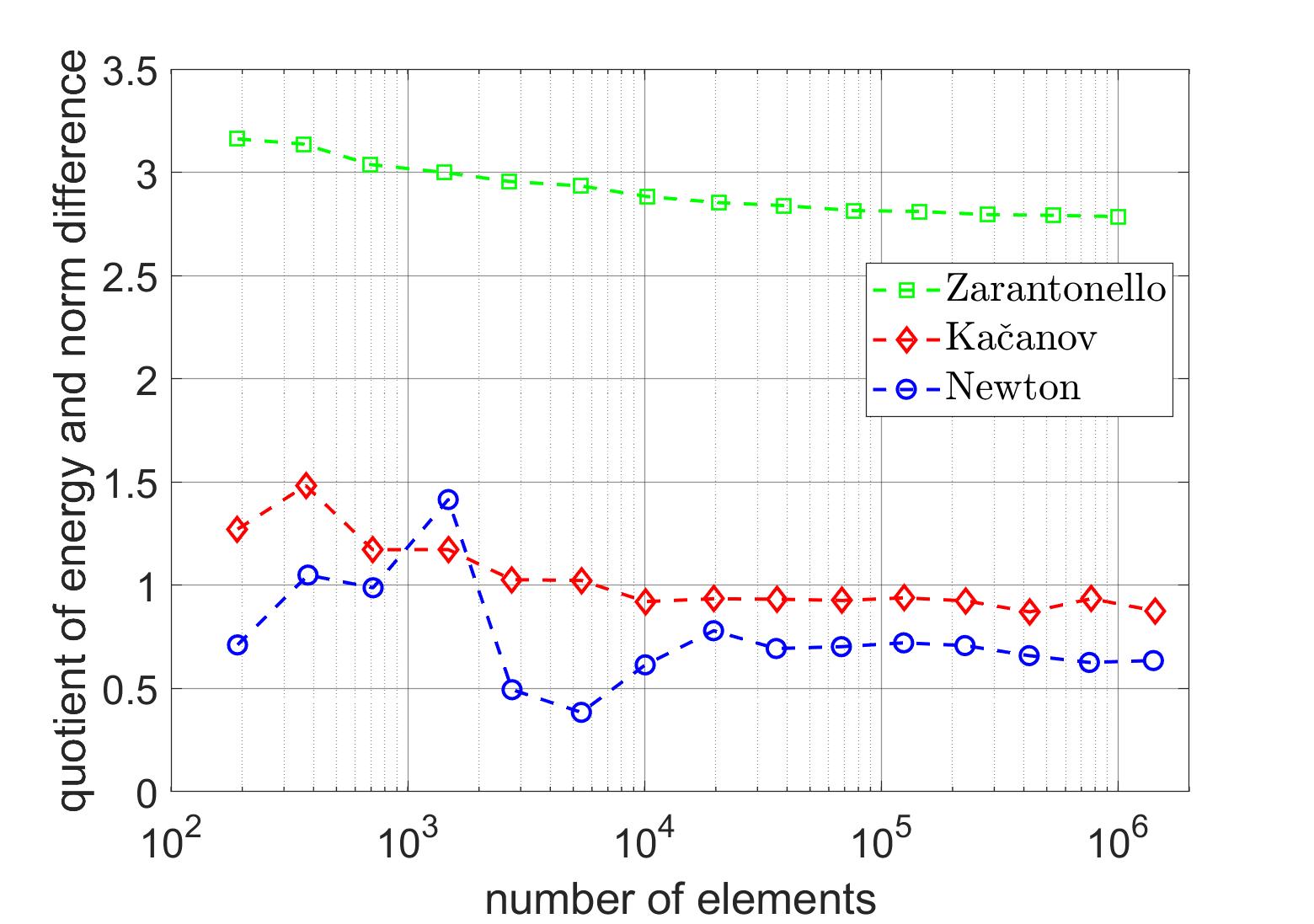}}
 \caption{The quotient $\kappa_N$ from~\eqref{eq:quotient} on each finite element space for $\delta_Z=0.1$ and $\lambda=0.5$ (left) and $\delta_Z=0.3$ and $\lambda=0.1$ (right), respectively.} \label{fig:quotient}
\end{figure}

\item $\delta_Z=0.3$ and $\lambda=0.1$: As before, in Figure~\ref{fig:NSC1003}, we display the performance of Algorithm~\ref{alg:praetal} with respect to both the number of elements and the total computational time. We clearly observe the optimal convergence rate of $-\nicefrac12$ for all of the three iteration schemes presented above from the initial mesh onwards. In contrast to the experiment before, the energy contraction factor~$\varkappa_N$ from~\eqref{eq:EC} is now of comparable size for all iteration schemes, as we can see from Figure~\ref{fig:NSCF1003} (left). Moreover, the number of iterations does not significantly differ for the three iterative methods. Again, we plot in Figure~\ref{fig:quotient} (right) the quotient~\eqref{eq:quotient} for the numerical evidence of the assumption (F4).

\begin{figure} 
 \subfloat{\includegraphics[width=0.49\textwidth]{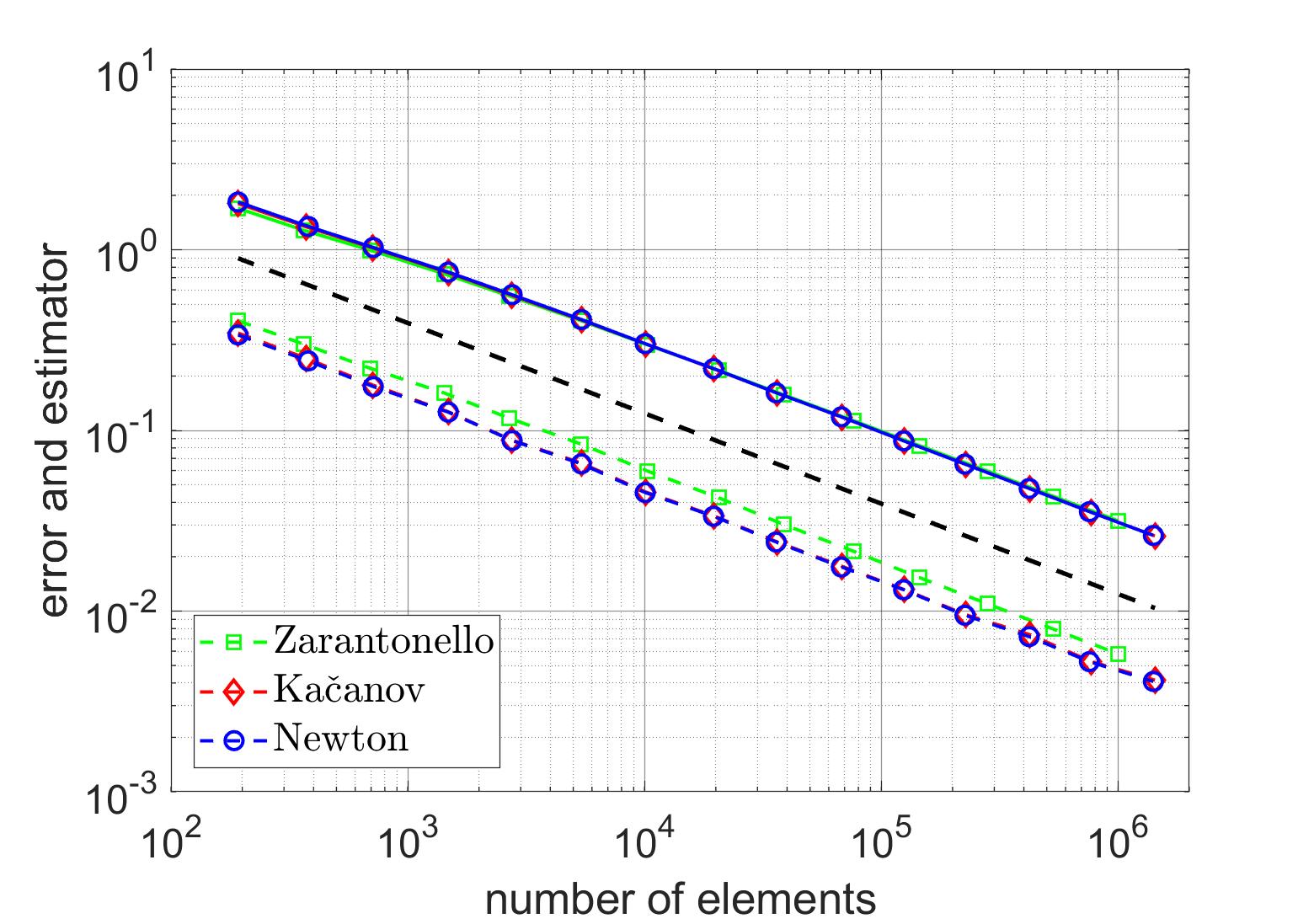}}\hfill
 \subfloat{\includegraphics[width=0.49\textwidth]{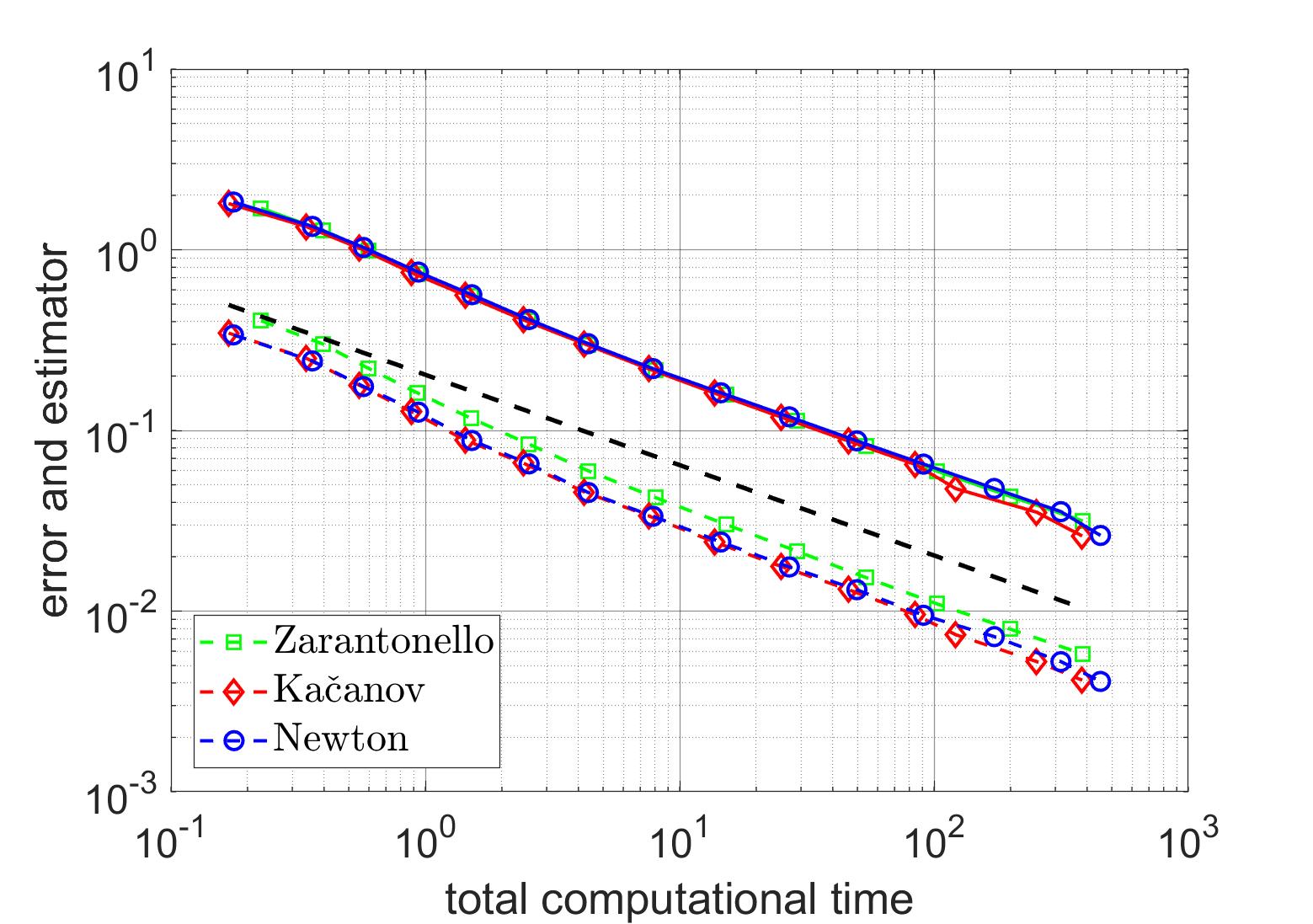}}
 \caption{$\delta_Z=0.3$ and $\lambda=0.1$: Performance plot for adaptively refined meshes with respect to the number of elements (left) and the total computational time (right). The solid and dashed lines correspond to the estimator and the error, respectively. The dashed lines without any markers indicate the optimal convergence order of $-\nicefrac12$ for linear finite elements.} \label{fig:NSC1003}
\end{figure}

\begin{figure} 
 \subfloat{\includegraphics[width=0.49\textwidth]{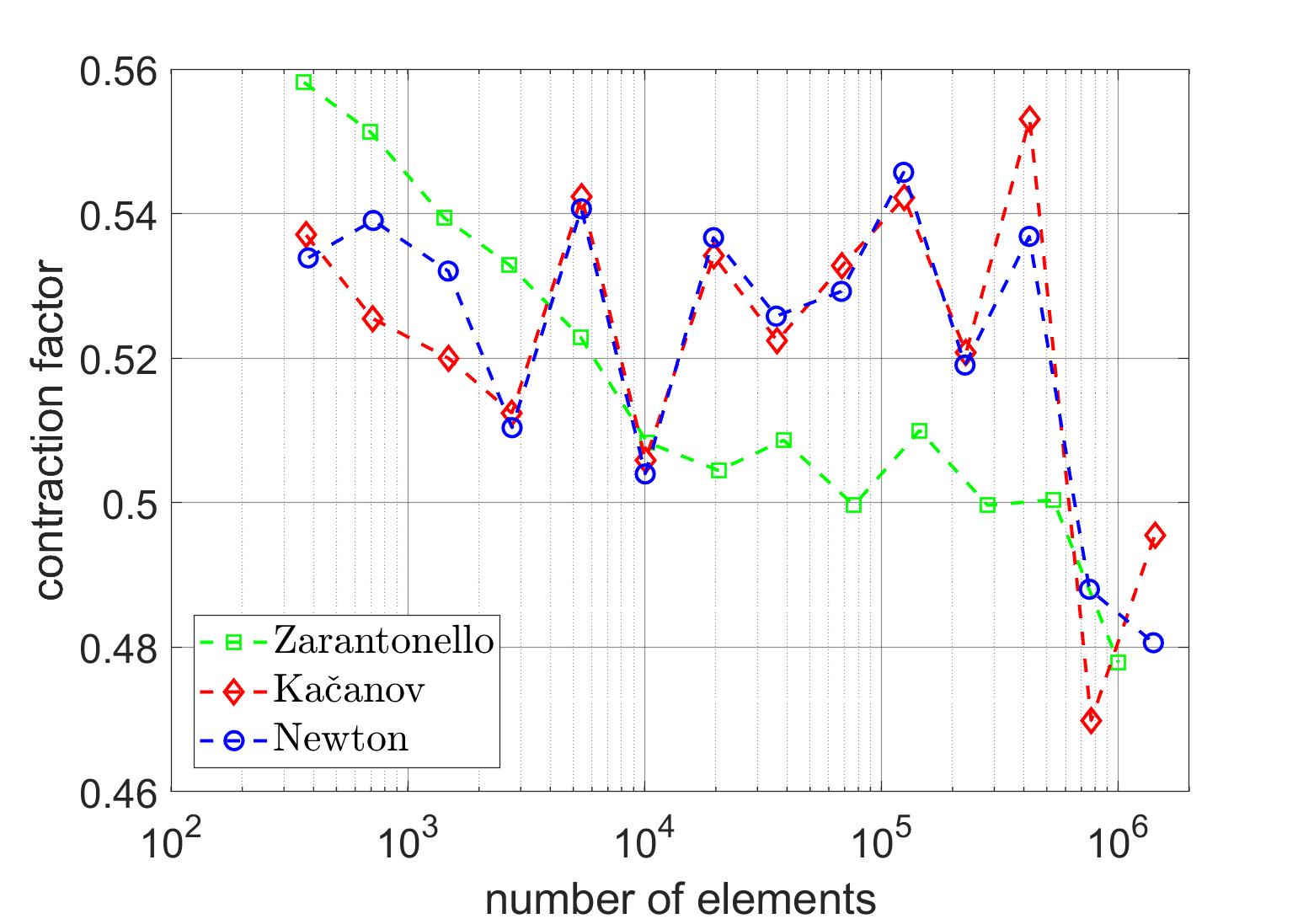}}\hfill
 \subfloat{\includegraphics[width=0.49\textwidth]{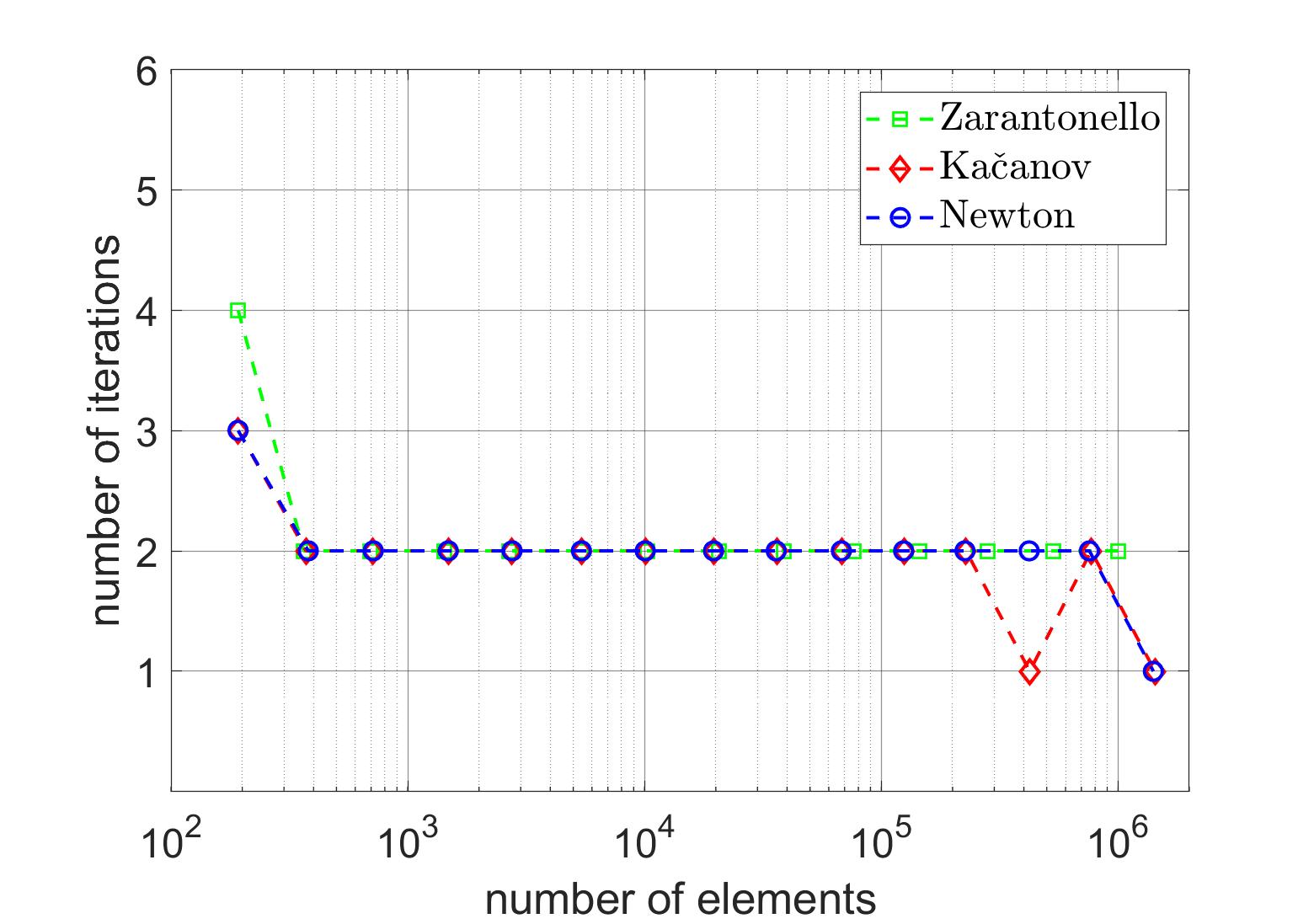}}
 \caption{$\delta_Z=0.3$ and $\lambda=0.1$: The contraction factor~$\varkappa_N$ (left, see~\eqref{eq:EC}) and the number of iterations (right) on each finite element space.} \label{fig:NSCF1003}
\end{figure}

\item $\delta_Z=0.3$ and $\lambda=0.01$: Once more, we observe optimal convergence rate for all our three iteration schemes with respect to both the number of elements and the total computational time, see Figure~\ref{fig:NSC10003}. The total computational times obtained, however, differ noticeably, see Figure~\ref{fig:NSC10003} (right), as a consequence of the varying number of iterative linearization steps of the three methods, see Figure~\ref{fig:NSCF10003} (right). In contrast, the energy contraction factor~$\varkappa_N$ from~\eqref{eq:EC} almost coincides for the different iteration schemes, see Figure~\ref{fig:NSCF10003} (left), which is due to the small adaptivity parameter $\lambda=0.01$. 

\begin{figure} 
 \subfloat{\includegraphics[width=0.49\textwidth]{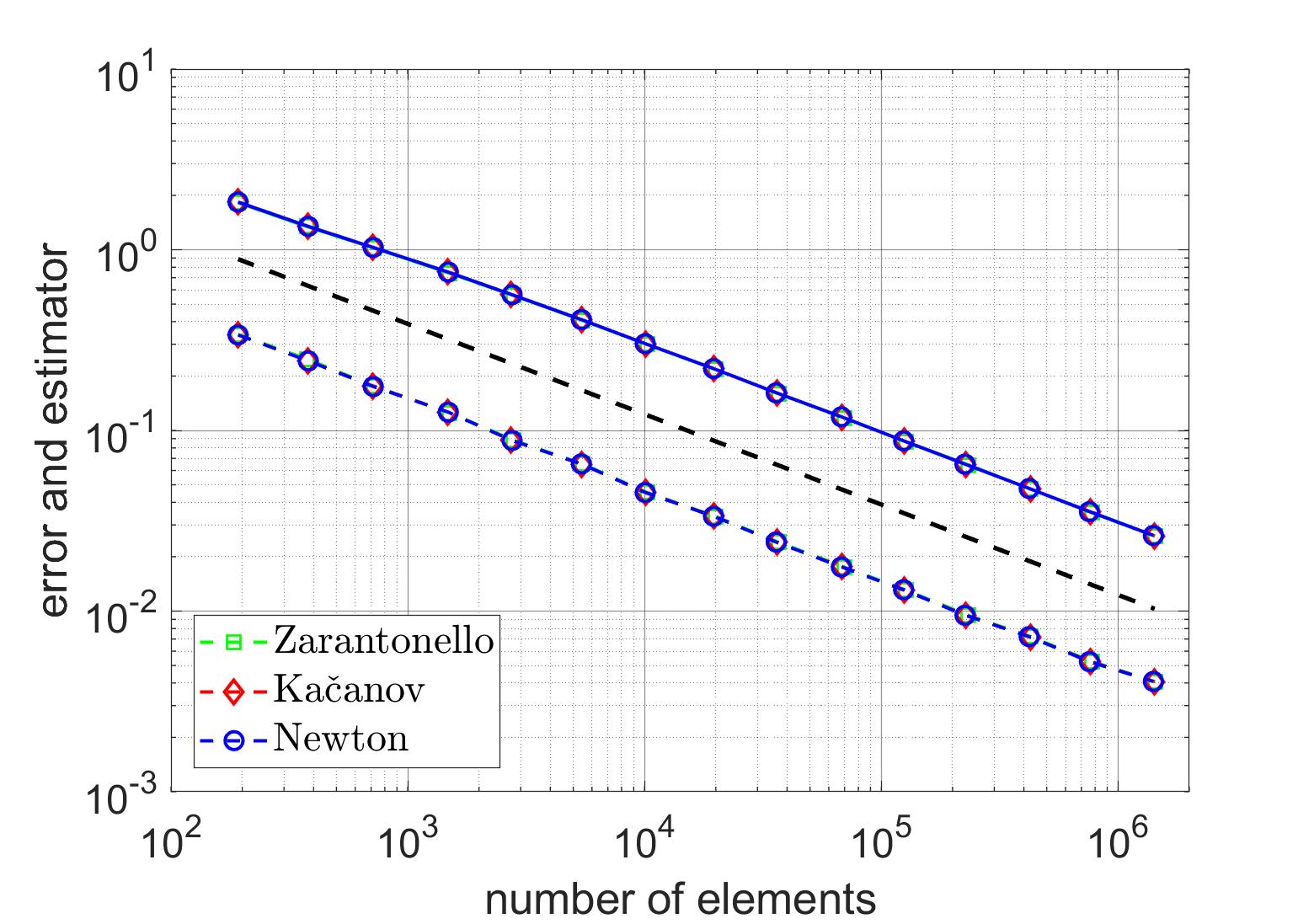}}\hfill
 \subfloat{\includegraphics[width=0.49\textwidth]{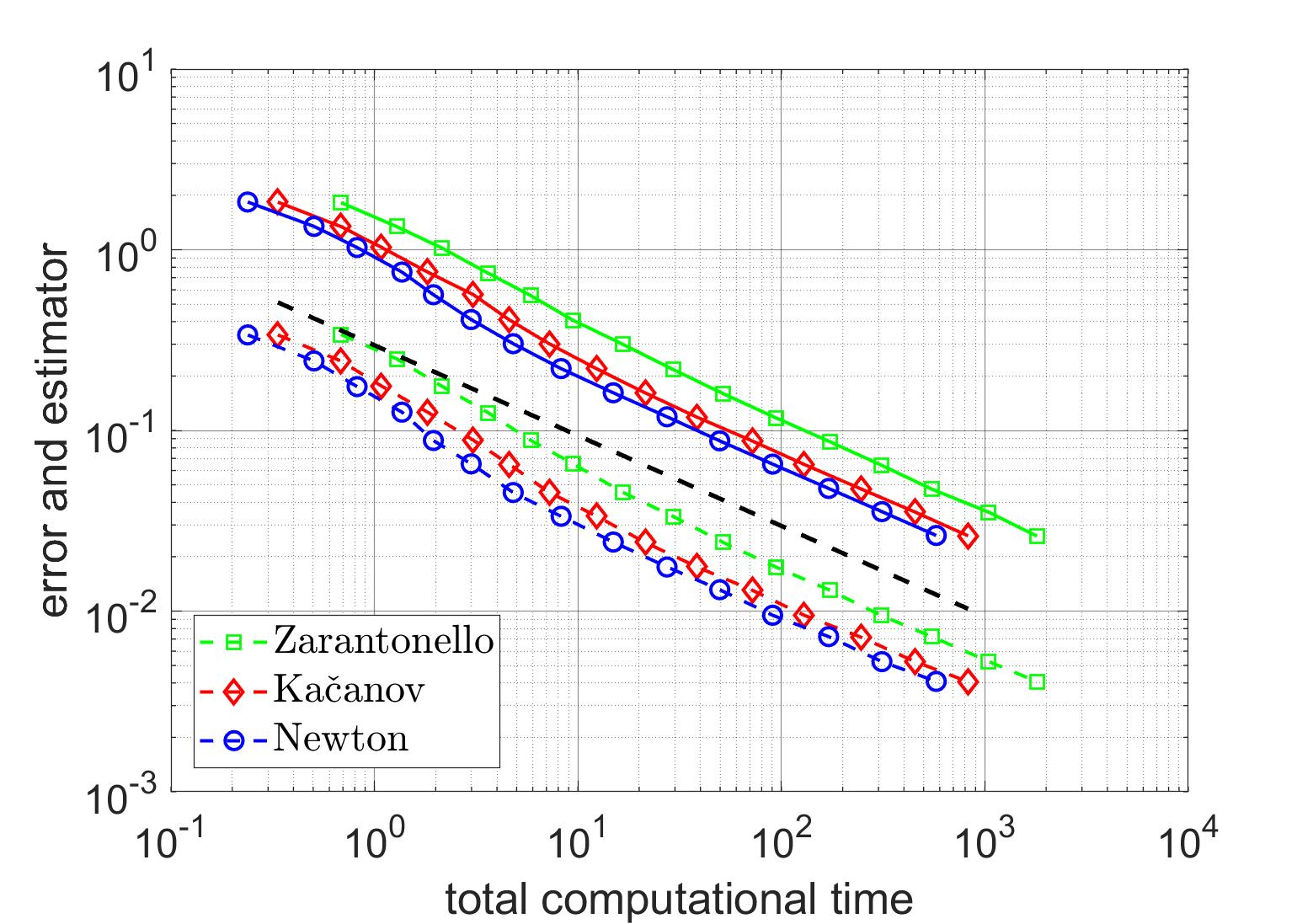}}
 \caption{$\delta_Z=0.3$ and $\lambda=0.01$: Performance plot for adaptively refined meshes with respect to the number of elements (left) and the total computational time (right). The solid and dashed lines correspond to the estimator and the error, respectively. The dashed lines without any markers indicate the optimal convergence order of $-\nicefrac12$ for linear finite elements.} \label{fig:NSC10003}
\end{figure}

\begin{figure} 
 \subfloat{\includegraphics[width=0.49\textwidth]{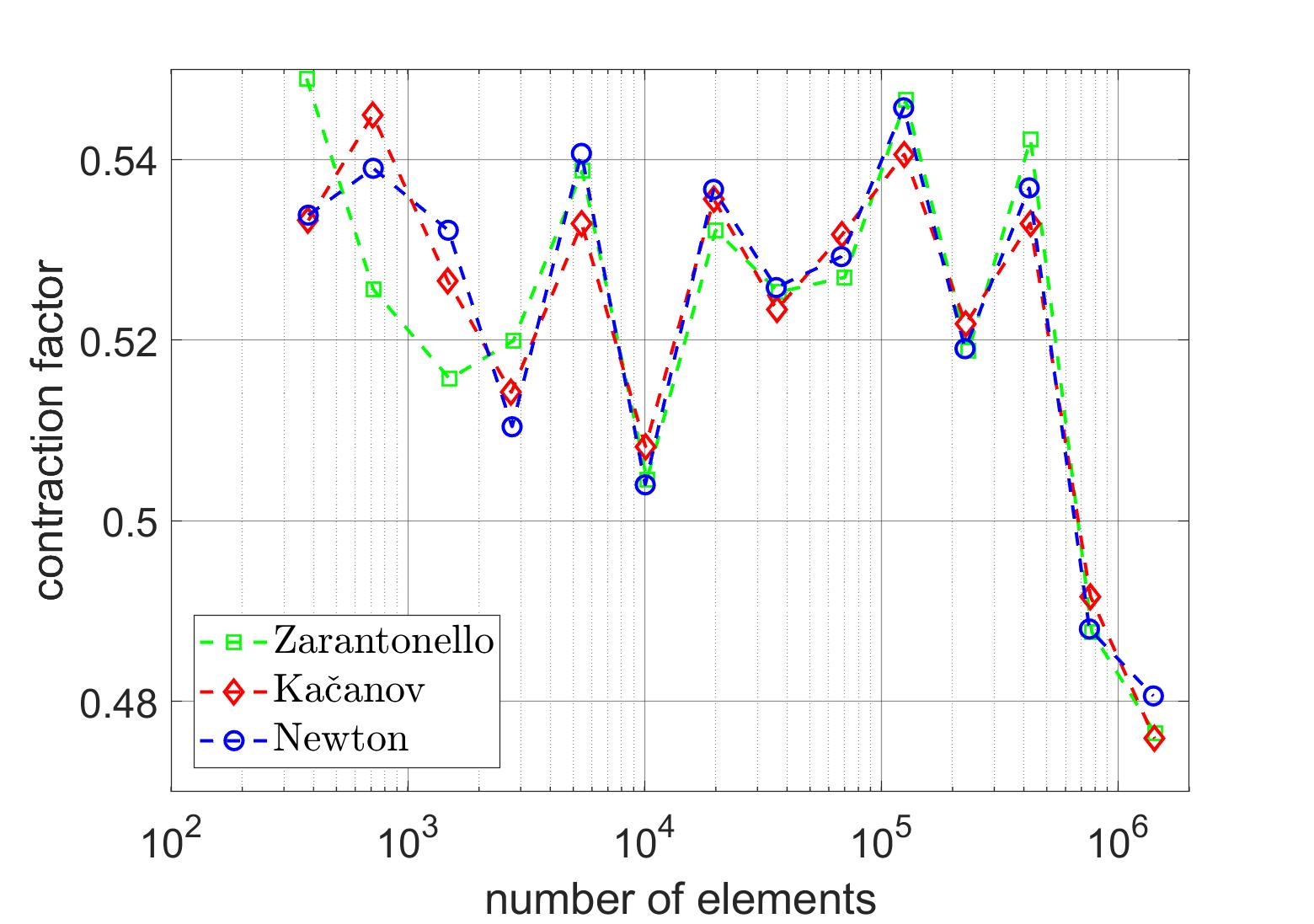}}\hfill
 \subfloat{\includegraphics[width=0.49\textwidth]{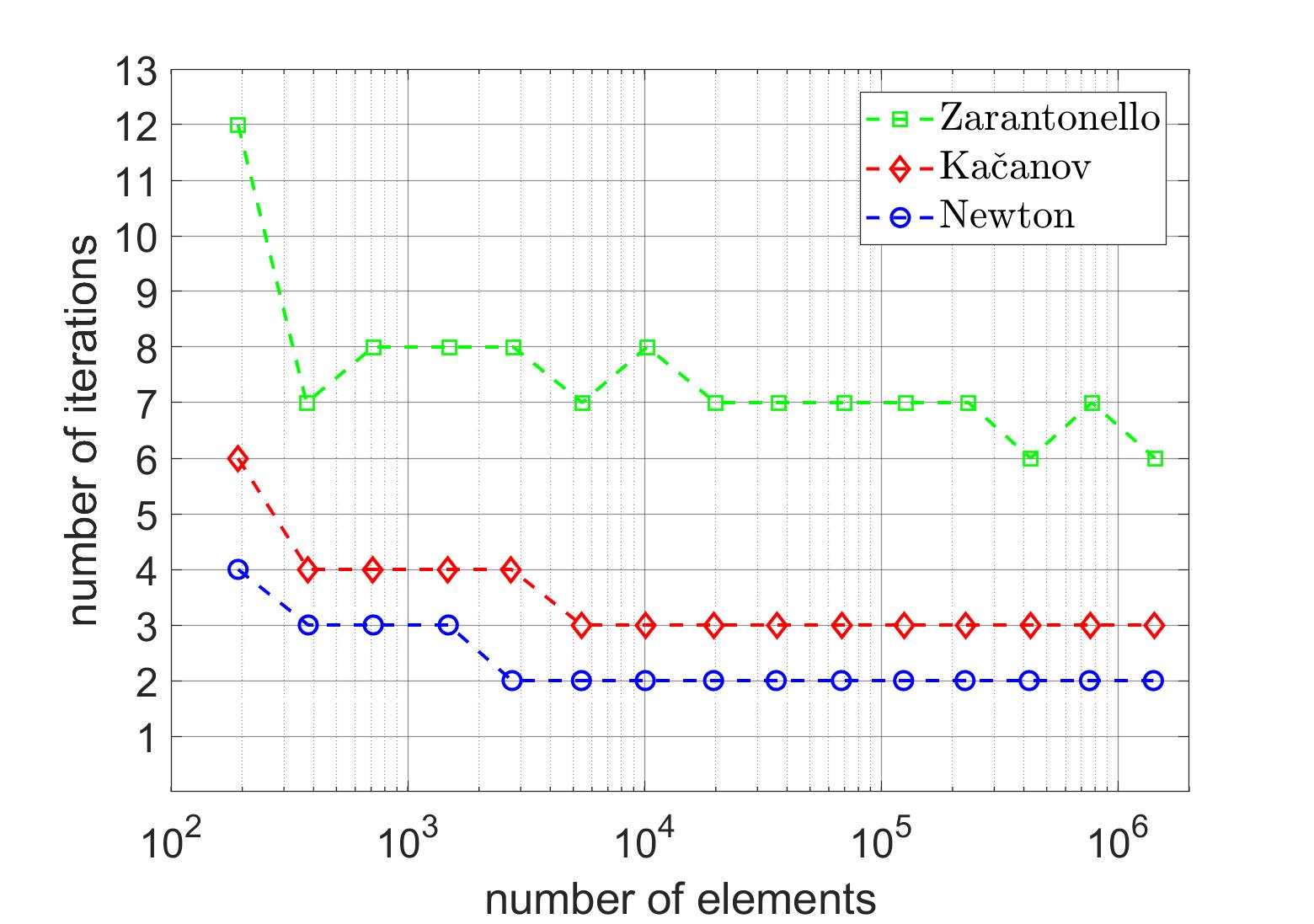}}
 \caption{$\delta_Z=0.3$ and $\lambda=0.01$: The contraction factor~$\varkappa_N$ (left, see~\eqref{eq:EC}) and the number of iterations (right) on each finite element space.} \label{fig:NSCF10003}
\end{figure}

\end{enumerate}

\section{Conclusions}\label{section:conclusion}

We have established a new energy contraction property for the  
iterative linearization Galerkin method (ILG, see~\eqref{eq:itweakY}). This result is the decisive prerequisite to apply the convergence theorems from~\cite{GHPS:2020}. In particular, the sequence generated by the adaptive iterative linearized finite element method (AILFEM, see Algorithm~\ref{alg:praetal}) converges always linearly with respect to both the adaptive linearization and the adaptive mesh-refinement, and, under some additional assumptions, even with optimal rate with respect to the overall computational cost. This is confirmed by our numerical test in the context of quasi-linear elliptic problems, where we also underline that the theoretical constraints on the adaptivity parameters $\lambda > 0$ and $0 < \theta \le 1$ are less restrictive in practice.

\bibliographystyle{amsalpha}
\bibliography{references}
\end{document}